\documentclass{amsart}
\usepackage{graphicx}

\newtheorem{thm}{Theorem}[section]
\newtheorem{cor}[thm]{Corollary}
\newtheorem{lem}[thm]{Lemma}
\newtheorem{prop}[thm]{Proposition}
\theoremstyle{definition}
\newtheorem{defn}[thm]{Definition}
\theoremstyle{remark}

\numberwithin{equation}{section} 

\newcommand{\A}{\mathcal{A}}
\begin{document}
\title[Tak\'acs' asymptotic theorem and its applications: A survey]
{Tak\'acs' asymptotic theorem and its applications: A survey}%
\author{Vyacheslav M. Abramov}%
\address{School of Mathematical Sciences, Monash University, Building
28M, Clayton Campus, Clayton, Victoria 3800}%
\email{vyacheslav.abramov@sci.monash.edu.au}%

\thanks{The research was supported by the Australian Research Council,
grant \#DP0771338}

\begin{abstract}
The book of Lajos Tak\'acs \emph{Combinatorial Methods in the Theory of
Stochastic
Processes} has been published in 1967. It discusses various problems
associated with
$$
P_{k,i}=\mathrm{P}\left\{\sup_{1\leq
n\leq\rho(i)}(N_n-n)<k-i\right\},\leqno(*)
$$
where $N_n=\nu_1+\nu_2\ldots+\nu_n$ is a sum of mutually independent,
nonnegative
integer and identically distributed random variables,
$\pi_j=\mathrm{P}\{\nu_k=j\}$,
$j\geq0$, $\pi_0>0$, and $\rho(i)$ is the smallest $n$ such that
$N_n=n-i$, $i\geq1$.
(If there is no such $n$, then $\rho(i)=\infty$.)

(*) is a discrete generalization of the classic ruin probability, and its
value is
represented as $P_{k,i}={Q_{k-i}}/{Q_k}$, where the sequence
$\{Q_k\}_{k\geq0}$
satisfies the recurrence relation of convolution type: $Q_0\neq0$ and
$Q_k=\sum_{j=0}^k\pi_jQ_{k-j+1}$.

Since 1967 there have been many papers related to applications of
the generalized classic ruin probability. The present survey
concerns only with one of the areas of application associated with
asymptotic behavior of $Q_k$ as $k\to\infty$. The theorem on
asymptotic behavior of $Q_k$ as $k\to\infty$ and further
properties of that limiting sequence are given on pages 22-23 of
the aforementioned book by Tak\'acs. In the present survey we
discuss applications of Tak\'acs' asymptotic theorem and other
related results in queueing theory, telecommunication systems and
dams. Many of the results presented in this survey have appeared
recently, and some of them are new. In addition, further
applications of Tak\'acs' theorem are discussed.
\end{abstract}

 \maketitle
\tableofcontents %
\section{Introduction}\label{Introduction}
The book of Tak\'acs \cite{Takacs 1967} has been published in 1967. It
discusses various
problems associated with
\begin{equation}\label{1.0}
P_{k,i}=\mathrm{P}\left\{\sup_{1\leq n\leq\rho(i)}(N_n-n)<k-i\right\},
\end{equation}
where $N_n=\nu_1+\nu_2\ldots+\nu_n$ is a sum of mutually independent,
nonnegative
integer and identically distributed random variables,
$\pi_j=\mathrm{P}\{\nu_k=j\}$,
$j\geq0$, $\pi_0>0$, and $\rho(i)$ is the smallest $n$ such that
$N_n=n-i$, $i\geq1$.
(If there is no such $n$, then $\rho(i)=\infty$.)

\eqref{1.0} is a discrete generalization of classic ruin probability, and
its value is
represented as $P_{k,i}=\frac{Q_{k-i}}{Q_k}$, where the sequence
$\{Q_k\}_{k\geq0}$
satisfies the recurrence relation of convolution type:
\begin{equation}\label{1.1}
Q_k=\sum_{j=0}^k\pi_jQ_{k-j+1},
\end{equation}
with arbitrary $Q_0\neq0$ (see Theorem 2 \cite{Takacs 1967} on page 18).
The probability
generating function of $Q_n$ is
\begin{equation}\label{1.2}
Q(z)=\sum_{k=0}^\infty Q_kz^k=\frac{Q_0\pi(z)}{\pi(z)-z},
\end{equation}
where $\pi(z)=\sum_{k=0}^\infty\pi_kz^k$.

There is also a continuous generalization of the classical ruin
theorem in \cite{Takacs 1967}.

Different applications of the aforementioned recurrence relation
have been provided for random walks, Brownian motion, queueing
processes, dam and storage processes, risk processes and order
statistics in \cite{Takacs 1967}. Most of these applications in
\cite{Takacs 1967} and following papers \cite{Takacs 1968} --
\cite{Takacs 1989} are associated with an explicit application of
a generalization of the ruin probability formula or its continuous
analogue.

The ruin problems and ballot theorems have a long history and
attract wide attention in the literature. The solution of the
first problems related to this subject had been given in 1887 due
to Bertrand \cite{Bertrand 1887} (see Sheinin \cite{Sheinin 1994}
for the review of Bertrand's work). For other papers related to
ballots problems, their generalizations and applications see
\cite{Heyde 1969}, \cite{Kemp and Kemp 1968}, \cite{Krattenthaler
and Mohanty 1994}, \cite{Lefevre 2007}, \cite{Lefevre and Loisel
2008}, \cite{Mazza and Rulliere 2004}, \cite{Mendelson 1982},
\cite{Takacs 1975}, \cite{Takacs 1997}, \cite{Tamaki 2001} and
many others.

The present survey is concerned with some special application of
recurrence relation \eqref{1.1}, where there is no explicit
application of the generalized gambler ruin problem. In our
application it does not matter what the meaning of values $\pi_j$
is, and what then the fraction $\frac{Q_{k-i}}{Q_k}$ means.

In this survey we discuss applications of asymptotic Theorem 5 of
Tak\'acs \cite{Takacs 1967} formulated on page 22 and the further
asymptotic results on page 23 associated with this theorem. We
also use the other asymptotic results (Tauberian theorems of
Postnikov which will be mentioned later). These asymptotic results
develop the aforementioned asymptotic theorem of Tak\'acs
\cite{Takacs 1967} and are used in the paper together with
Tak\'acs' theorem to accomplish it. Nevertheless, the title of
this paper is supported by the fact that Tak\'acs' theorem serves
as a tool to establish the main asymptotic properties of
processes, while Postnikov's Tauberian theorems improve and
strengthen the results on that asymptotic behavior, when
additional conditions are satisfied. In some delicate cases,
applications of Postnikov's Tauberian theorems have especial
significance, when the case of Tak\'acs' theorem cannot help to
establish a required property.

For our convenience, we use the variant of Theorem 5 of \cite{Takacs
1967} in
combination with formula (35) on page 23 as follows. Denote
$$
\gamma_\ell=\sum_{j=\ell}^\infty\prod_{k=1}^\ell (j-k+1)\pi_j.
$$
The value $\gamma_\ell$ characterizes the $\ell$th factorial
moment related to the probabilities $\pi_j$, $j\geq0$.

\begin{thm}
\label{thmTakacs} Let $\pi_0>0$. If $\gamma_1<1$, then
\begin{equation}\label{1.4}
\lim_{k\to\infty}Q_k=\frac{Q_0}{1-\gamma_1}.
\end{equation}
If $\gamma_1=1$ and $\gamma_2<\infty$, then
\begin{equation}
\label{1.5} \lim_{k\to\infty}\frac{Q_k}{k}=\frac{2Q_0}{\gamma_2}.
\end{equation}
If $\gamma_1>1$, then
\begin{equation}
\label{1.6}
\lim_{k\to\infty}\left[Q_k-\frac{Q_0}{\sigma^k(1-\pi^\prime(\sigma))}\right]
=\frac{Q_0}{1-\gamma_1},
\end{equation}
where $\sigma$ is the least nonnegative root of the equation $z=\pi(z)$,
and
$0<\sigma<1$.
\end{thm}

In the case $\gamma_2=\infty$, \cite{Takacs 1967} recommends to
use the Tauberian theorem of Hardy and Littlewood \cite{Hardy and
Littlewood 1929}, \cite{Hardy 1949}, which states that if
$$
Q(z)\asymp\frac{1}{(1-z)^{\alpha+1}}L\left(\frac{1}{1-z}\right),
$$
as $z\uparrow1$, where $\alpha\geq0$ and $L(x)$ is a slowly varying
function at infinity
(i.e. $L(cx)\asymp L(x)$ for any positive $c$ as $x\to\infty$), then
$$
Q_k\asymp\frac{k^\alpha}{\Gamma(\alpha+1)}L(k),
$$
as $k\to\infty$, where $\Gamma(x)$ is the Euler gamma-function.

The case $\gamma_1=1$ and $\gamma_2<\infty$ has been developed by
Postnikov \cite{Postnikov 1980}, Sect. 25. He established the
following two Tauberian theorems.

\begin{thm}
\label{thm1Postnikov} Let $\pi_0>0$. Suppose that $\gamma_1=1$ and
$\gamma_3<\infty$.
Then, as $k\to\infty$,
\begin{equation}
\label{1.9} Q_k=\frac{2Q_0}{\gamma_2}~n+O(\log n).
\end{equation}
\end{thm}

\begin{thm}
\label{thm2Postnikov} Let $\pi_0>0$. Suppose that $\gamma_1=1$,
$\gamma_2<\infty$ and
$\pi_0+\pi_1<1$. Then, as $k\to\infty$,
\begin{equation}
\label{1.10} Q_{k+1}-Q_k=\frac{2Q_0}{\gamma_2}+o(1).
\end{equation}
\end{thm}

The applications reviewed in this survey are related to queueing
theory, telecommunication systems and dams/storage systems.

The first paper in this area has been published in 1981 by the
author \cite{Abramov 1981}. The asymptotic behavior of the mean
busy period of the $M/G/1$ queueing system with service depending
on queue-length has been studied there. This and the several later
papers have been summarized in book \cite{Abramov 1991}. Other
research papers, associated with application of \eqref{1.1} have
appeared recently \cite{Abramov 1997}, \cite{Abramov 2002} --
\cite{Abramov 2007 - water}. In this survey we also discuss many
other closely related papers, the results of which have been
obtained by many authors. Several new results related to the
asymptotic behavior of consecutive losses in $M/GI/1/n$ queueing
systems are also established (Sect. \ref{Consecutive losses}), and
they then applied for analysis of consecutive losses in
telecommunication systems (Sect. \ref{Lost messages - new}).

The survey is organized as follows. In Sect. 2-6, which are the
first part of this survey, we discuss the classical loss queueing
systems such as $M/GI/1/n$, $GI/M/1/n$ and $GI/M/m/n$. For these
systems we discuss applications of Tak\'acs' theorem
\ref{thmTakacs} and Postnikov's theorems \ref{thm1Postnikov} and
\ref{thm2Postnikov} for analysis of the limiting stationary loss
probabilities. Sect. \ref{Consecutive losses} contains new results
on consecutive losses. In Sect. 6 we discuss further research
problems related to asymptotic analysis of large retrial queueing
systems.

In Sect. 7-12, which, respectively, are the second part of this
paper, we discuss various applications of these theorems for
different models of telecommunication systems and dams/storage
systems. All of these applications of Tak\'acs' theorem are not
traditional, and, being specific, the models considered here
describe real-world systems or models for these systems. The
results of Sect. \ref{Lost messages - new} are new. They are based
on applications of new results on consecutive losses, presented in
Sect. \ref{Consecutive losses}. In Sect. \ref{Dam3} we discuss
further research problems and the ways of their solution
associated with asymptotic analysis and optimal control of large
dams presented in Sect. \ref{Dam1} and \ref{Dam2}.

\part{Applications to queueing systems}

\section{Losses in the $M/GI/1/n$ queueing system}\label{MGI1n queue}

Considering $M/GI/1/n$ queue, Tomk\'o \cite{Tomko 1967} was the
first person to establish recurrence relation \eqref{1.1} for the
expected busy periods. (We assume that $n$ is the number of
waiting places excluding the service space.) Let $T_n$ denote a
busy period, and let $T_k$, $k=0,1,\ldots,n-1$, be busy periods
associated with similar queueing systems (i.e. having the same
rate $\lambda$ of Poisson input and the same probability
distribution function $B(x)$ of a service time) but only having
different number of waiting places. Namely Tomk\'o \cite{Tomko
1967} has derived
\begin{equation}\label{2.1}
\mathrm{E}T_n=\sum_{j=0}^{n}\mathrm{E}(T_{n-j+1})\int_0^\infty\mathrm{e}^{-\lambda
x}\frac{(\lambda x)^j}{j!}\mathrm{d}B(x).
\end{equation}
In the case of deterministic service times, Tomk\'{o} \cite{Tomko 1967}
used Karamata's
Tauberian theorem \cite{Widder 1941}, Chapt. 5.

Representation \eqref{2.1} has been also obtained by Cohen \cite{Cohen
1971}, Cooper and
Tilt \cite{Cooper and Tilt 1976}, Rosenlund \cite{Rosenlund 1973} as well
as by the
author in \cite{Abramov 1981} and \cite{Abramov 1991}.

The relationship between the mean busy period in the $M/GI/1/n$ queueing
system and the
maximum queue-length distribution in the $M/GI/1$ queueing system has
been established
by Cooper and Tilt \cite{Cooper and Tilt 1976}. Namely, they derived
Tak\'acs' result
\cite{Takacs 1969-2} on the supremum queue-length distribution of the
$M/G/1$ queue
during a busy period by the methods distinguished from the original
Tak\'{a}cs' method.

With the aid of Theorems \ref{thmTakacs}, \ref{thm1Postnikov} and
\ref{thm2Postnikov}
one can easily study the asymptotic behavior of $\mathrm{E}T_n$ as
$n\to\infty$.
Immediately from these theorems we correspondingly have as follows.
\begin{thm}\label{thm1-bp} If $\rho<1$, then
$$
\lim_{n\to\infty}\mathrm{E}T_n=\frac{\rho}{\lambda(1-\rho)}.
$$
If $\rho=1$ and $\rho_2=\lambda^2\int_0^\infty x^2\mathrm{d}B(x)<\infty$,
then
$$
\lim_{n\to\infty}\frac{\mathrm{E}T_n}{n}=\frac{2}{\lambda\rho_2}.
$$
If $\rho>1$, then
\begin{equation}\label{ET_n rho geq1}
\lim_{n\to\infty}\left\{\mathrm{E}T_n-\frac{\rho}{\lambda[1+\lambda\widehat{B}^\prime(\lambda-\lambda\varphi)]
\varphi^n}\right\}=\frac{\rho}{\lambda(1-\rho)},
\end{equation}
where $\widehat{B}(s)=\int_0^\infty\mathrm{e}^{-sx}\mathrm{d}B(x)$ is the
Laplace-Stieltjes
transform of a service time, and $\varphi$ is the least positive root of
the equation
\begin{equation}\label{2.10}
 z=\widehat{B}(\lambda-\lambda z).
\end{equation}
\end{thm}

\begin{thm}\label{thm2-bp} If $\rho=1$ and $\int_0^\infty
x^3\mathrm{d}B(x)<\infty$, then as $n\to\infty$
$$
\mathrm{E}T_n=\frac{2}{\lambda\rho_2}n+O(\log n).
$$
\end{thm}

\begin{thm}\label{thm3-bp} If $\rho=1$ and $\rho_2<\infty$, then
$$
\mathrm{E}T_n-\mathrm{E}T_{n-1}=\frac{2}{\lambda\rho_2}+o(1).
$$
\end{thm}

The proof of Theorems \ref{thm1-bp} and \ref{thm2-bp} follows immediately
from the
statements of Theorems \ref{thmTakacs} and \ref{thm1Postnikov}
correspondingly. Let us
prove Theorem \ref{thm3-bp}.

\begin{proof}
To prove this theorem, it is enough to show that for all $\lambda>0$
\begin{equation}
\label{3.13} \widehat{B}(\lambda)-\lambda\widehat{B}^\prime(\lambda)<1.
\end{equation}
Then all of the conditions of Theorem \ref{thm2Postnikov} will be
satisfied, and the
statement of this theorem will follow immediately from Theorem
\ref{thm2Postnikov}.

Let us prove \eqref{3.13}. Taking into account that
$$
\sum_{j=0}^\infty\int_0^\infty\mathrm{e}^{-\lambda x}\frac{(\lambda
x)^j}{j!}\mathrm{d}B(x)=\int_0^\infty\sum_{j=0}^\infty\mathrm{e}^{-\lambda
x}\frac{(\lambda
x)^j}{j!}\mathrm{d}B(x)=1,
$$
and
$$
\int_0^\infty\mathrm{e}^{-\lambda x}\frac{(\lambda
x)^j}{j!}\mathrm{d}B(x)\geq0
$$
we have
\begin{equation*}
\label{3.16}
\widehat{B}(\lambda)-\lambda\widehat{B}^\prime(\lambda)\leq1.
\end{equation*}
Therefore \eqref{3.13} will be proved if we show that for some
$\lambda_0$ the equality
\begin{equation}
\label{3.17}
\widehat{B}(\lambda_0)-\lambda_0\widehat{B}^\prime(\lambda_0)=1
\end{equation}
is not the case. Indeed, since
$\widehat{B}(\lambda)-\lambda\widehat{B}^\prime(\lambda)$
is an analytic function in $\lambda$, then, according to the theorem on
maximum absolute
value of an analytic function, equality \eqref{3.17} is to valid for all
$\lambda>0$,
i.e. $\widehat{B}(\lambda)-\lambda\widehat{B}^\prime(\lambda)=1$. This
means that
\eqref{3.17} is valid if and only if $\int_0^\infty\mathrm{e}^{-\lambda
x}\frac{(\lambda
x)^j}{j!}\mathrm{d}B(x)=0$ for all $j\geq2$ and for all $\lambda>0$. This
in turn
implies that $\widehat{B}(\lambda)$ must be a linear function, i.e.
$\widehat{B}(\lambda)=c_0+c_1\lambda$ for some constants $c_0$ and $c_1$.
However,
because of $|\widehat{B}(\lambda)|\leq1$, we obtain $c_0=1$ and $c_1=0$.
This
corresponds to the trivial case, where the probability distribution
function $B(x)$ is
concentrated in point 0. This case cannot be valid, since $\lambda$ is
assumed to be
strictly positive. Therefore \eqref{3.17} is not the case, and hence
\eqref{3.13}
holds.
\end{proof}

Let $L_n$ denote the number of losses during the busy period $T_n$, and
let $\nu_n$
denote the number of served customers during the same busy period.

Using Wald's equations \cite{Feller 1966}, p.384, we have the following
system.
\begin{eqnarray}
\mathrm{E}\nu_n+\mathrm{E}L_n&=&\lambda\mathrm{E}T_n+1,\label{2.2}\\
\mathrm{E}\nu_n&=&\mu\mathrm{E}T_n,\label{2.3}
\end{eqnarray}
where $\mu$ denotes the reciprocal of the expected service time. The term
$\lambda\mathrm{E}T_n$ of the right-hand side of \eqref{2.2} denotes the
expected number
of arrivals during a regeneration period excluding the first (tagged)
customer, who
starts the busy period. Equations \eqref{2.2} and \eqref{2.3} are similar
to those (1)
and (2) of paper \cite{Abramov 2001}.

From these two equations \eqref{2.2} and \eqref{2.3}, by subtracting the
second from the
first, we obtain:
\begin{equation*}
\label{2.4} \mathrm{E}L_n=(\lambda-\mu)\mathrm{E}T_n+1,
\end{equation*}
leading to
\begin{equation}
\label{2.5} \mathrm{E}L_n-1=(\lambda-\mu)\mathrm{E}T_n.
\end{equation}
Equation \eqref{2.5} coincides with (7) of \cite{Abramov 1997} which is
written under
another notation and was derived by the different method. From
\eqref{2.5} we obtain the
following recurrence relation:
\begin{equation}
\label{2.6}
\mathrm{E}L_n-1=\sum_{j=0}^n(\mathrm{E}L_{n-j+1}-1)\int_0^\infty\mathrm{e}^{-\lambda
x}\frac{(\lambda x)^j}{j!}\mathrm{d}B(x), \ \
\mathrm{E}L_0=\frac{\lambda}{\mu}=\rho.
\end{equation}
Recurrence relation \eqref{2.6} is the same as (3) of \cite{Abramov 1997}
and enables us
to prove the following theorem (see \cite{Abramov 1984}, \cite{Abramov
1991} and
\cite{Abramov 1997}).

\begin{thm}
\label{thm1} If $\rho<1$, then
\begin{equation*}\label{2.7}
\lim_{n\to\infty}\mathrm{E}L_n=0.
\end{equation*}
If $\rho=1$, then for all $n\geq0$
\begin{equation}
\label{2.8} \mathrm{E}L_n=1.
\end{equation}
If $\rho>1$, then
\begin{equation}
\label{2.9}
\lim_{n\to\infty}\left[\mathrm{E}L_n-\frac{(\rho-1)\varphi^{-n}}
{1+\lambda
\widehat{B}^\prime(\lambda-\lambda\varphi)}\right]=0,
\end{equation}
where $\widehat{B}(s)=\int_0^\infty\mathrm{e}^{-sx}\mathrm{d}B(x)$ is the
Laplace-Stieltjes
transform of a service time, and $\varphi$ is the least positive root of
the equation
\eqref{2.10}.
\end{thm}

Evidently, that the statement of this theorem follows immediately by
application of
Theorem \ref{thmTakacs}. In this case we equate $\pi_j$ with
$$\int_0^\infty\mathrm{e}^{-\lambda x}\frac{(\lambda
x)^j}{j!}\mathrm{d}B(x),$$ and the
condition $\pi_0>0$ is equivalent to $\widehat{B}(\lambda)>0$ and is
satisfied. Then we
equate $\gamma_1$ with $\rho$ and $Q_k$ with $\mathrm{E}L_k-1$ and easily
arrive at the
statement of the theorem, where under the condition $\rho>1$, there is a
unique root of
equation \eqref{2.10} belonging to the interval (0,1) (e.g. Tak\'acs
\cite{Takacs
1962}).

The most significant consequence of this theorem is \eqref{2.8}. This
result met
attention in the literature, and there is a number of papers that devoted
to extension
of this result. We mention \cite{Abramov 2001}, \cite{Pekoz 1999},
\cite{Righter 1999},
\cite{Pekoz Righter and Xia 2003}, \cite{Wolff 2002} as well as
\cite{Abramov 2007-3}.
On the other hand, rel. \eqref{2.8} is an unexpected generalization of
the following
elementary result from random walk theory\footnote{The result is not
widely known. The
mention about this result can be found, for example, in the book of
Szekely
\cite{Szekely 1986}. For another relevant consideration see Wolff \cite{Wolff 1989}, p. 411}. (This random walk is associated with the Markovian
$M/M/1/n$
queueing system.) Namely, let $X_0$, $X_1$, \ldots be random variables
taking the values
$\pm1$ with equal probability $\frac{1}{2}$. Let $S_0=0$. Then
$S_n=X_1+X_2+\ldots+X_n$
characterizes a symmetric random walk starting at zero, and
$\tau=\inf\{l>0: S_l=0\}$ is
the moment of the first return to zero point. Then for any level
$m\neq0$, the
expectation of the number of up-crossings across the level $m$ during the
random time
$[0, \tau]$ is equal to $\frac{1}{2}$.

For $\rho>1$, the asymptotic analysis of $\mathrm{E}T_n$ and
$\mathrm{E}L_n$ of the $M/GI/1/n$ queueing system has also been
provided in Azlarov and Takhirov \cite{Azlarov and Tahirov 1974}
and Azlarov and Tashmanov \cite{Azlarov and Tashmanov 1974} by
methods of the complex analysis. The similar results for the
$GI/M/1/n$ queue have been obtained by Takhirov \cite{Tahirov
1980}. The asymptotic estimations obtained in these papers have
the worse order of the remainder than that following from
Tak\'acs' theorem \ref{thmTakacs}.

Theorem \ref{thm1} has been extended in \cite{Abramov 1991-2} for the
case of the
$M/GI/1$ queueing system where if an arriving customer meets $n$ or more
customers in
the queue, then he joins the queue with probability $p$ or lost with
probability
$q=1-p$.

The case $\rho=1+\delta$ where $\delta n\to C\geq0$ as $\delta\to0$ and
$n\to\infty$
falls into the area of the heavy traffic analysis. Under this assumption
we have the
following theorem (see, Theorem 2 of \cite{Abramov 1997}).

\begin{thm}\label{thm2}
Let $\rho=1+\delta$, $\delta>0$. Denote $\rho_j(\delta)=\int_0^\infty
(\lambda
x)^j\mathrm{d}B(x)$. Assume that $\delta n\to C\geq0$ as $\delta>0$ and
$n\to\infty$.
Assume also that $\widetilde{\rho}_2=\lim_{n\to \infty}\rho_2(\delta)$
and $\rho_3(n)$
is a bounded sequence. Then,
\begin{equation}
\label{2.11} \mathrm{E}L_n=\mathrm{e}^{2C/\widetilde{\rho}_2}[1+O(\delta)].
\end{equation}
\end{thm}

The proof of Theorem \ref{thm2} uses the expansion for small $\delta$
\begin{equation}
\label{2.12} \varphi=1-\frac{2\delta}{\widetilde{\rho}_2}+O(\delta^2).
\end{equation}
obtained in Subhankulov \cite{Subhankulov 1976}, p. 326 as well as the
expansion
\begin{equation}
\label{2.13}
1+\lambda\widehat{B}^\prime(\lambda-\lambda\varphi)=\delta+O(\delta^2),
\end{equation}
following from the Taylor expansion of
$\widehat{B}^\prime(\lambda-\lambda\varphi)$.

Using renewal theory and result \eqref{2.12}, one can obtain the
estimation for the
stationary loss probability for the $M/GI/1/n$ queueing system with large
number of
servers, and the load parameter $\rho=1+\delta$, $\delta>0$, as $\delta
n\to C\geq0$.

Let $\A_n$ denote the number of arrived customers during the busy cycle.
According to
renewal theory, for the stationary loss probability $P_{loss}$ we have:
\begin{equation*}
P_{loss}=\frac{\mathrm{E}L_n}{1+\mathrm{E}A_n}
=\frac{\mathrm{E}L_n}{1+\lambda\mathrm{E}T_n}.
\end{equation*}
Under the assumption $\rho>1$, from Theorem \ref{thmTakacs} we have the
following
expansion for $\mathrm{E}A_n$:
\begin{equation}\label{2.15}
\mathrm{E}A_n=\frac{{\rho}\varphi^{-n}}{1+\lambda\widehat{B}^\prime(\lambda-\lambda\varphi)}+
\frac{\rho}{1-\rho}+o(1).
\end{equation}
Therefore under the assumption $\rho>1$, for $P_{loss}$ we obtain:
\begin{equation}
\label{2.16} P_{loss}=\frac{(\rho-1)^2}{\rho(\rho-1)-\varphi^n
[1+\lambda\widehat{B}^\prime(\lambda-\lambda\varphi)]}[1+o(\varphi^n)].
\end{equation}

Under the assumption that $\rho=1+\delta$, $\delta>0$, and $\delta n\to
C\geq0$ we have
the following result.

\begin{thm}
\label{thm3} Under the assumptions of Theorem \ref{thm2}, the following
asymptotic
relation holds:
\begin{equation}
\label{2.17} P_{loss} =
\frac{\delta}{1+\delta-\mathrm{e}^{-2C/\widetilde{\rho}_2}}[1+O(\delta)].
\end{equation}
\end{thm}

Similar asymptotic relation for $M/GI/1/n$ queues, under other heavy
traffic conditions
where $\delta$ is a small negative parameter, has been discussed by Whitt
\cite{Whitt
2004}.

\section{Consecutive losses in $M/GI/1/n$ queues}\label{Consecutive
losses}

Consecutive losses are very important in performance analysis of real
telecommunication
systems. They have been studied in many papers, see \cite{Chydzinski
2004},
\cite{Chydzinski 2005}, \cite{Chydzinski 2007}, \cite{DeBoer et al 2001},
\cite{Pacheco
and Ribeiro 2006}, \cite{Pacheco and Ribeiro 2008}, \cite{Pacheco and
Ribeiro 2008b} and
many others. A convolution type recurrence relation for the distribution
of $k$
consecutive losses during a busy period (or stationary probability of $k$
consecutive
losses) has been obtained in \cite{Pacheco and Ribeiro 2008} for more
general systems
with batch arrivals. Although the structure of these recurrence relations
is simple, the
notation used there looks complicated. Unfortunately, an asymptotic
analysis of these
recurrence relations or those ones for other numerical characteristics,
as $n\to\infty$,
has not been provided neither for the systems with batch arrivals that
considered in
this paper \cite{Pacheco and Ribeiro 2008}, nor for the simpler
$M/GI/1/n$ queueing
systems.

We derive now the relation similar to \eqref{2.6} to the case of
consecutive losses in $M/GI/1/n$ queueing systems. By $k$-CCL
probability\footnote{This terminology is used in \cite{Pacheco and
Ribeiro 2008} and \cite{Pacheco and Ribeiro 2008b}.} we mean the
(limiting) fraction of those losses belonging to sequences of $k$
or more consecutive losses (we call them $k$-consecutive losses)
with respect to all of the losses that occur during a long time
interval. Let $E_1$, $E_2$, \ldots, $E_{n}$ be the states of the
$M/GI/1/n$ queueing system ($n\geq1$ and the notation for the
$M/GI/1/n$ queue excludes a customer in service) at the moments of
service begins, i.e. states when at the moments of service begins
there are 1, 2, \ldots, $n$ customers in the system
correspondingly. Suppose that the system is currently in the state
$E_j$ ($1\leq j\leq n$). Then, a loss of a customer (at least one)
occurs in the case when during the service time of the tagged
customer, which is currently in service, there are $n-j+2$ or more
new arrivals. This probability is equal to
$$
\sum_{i=n-j+2}^\infty\int_0^\infty\mathrm{e}^{-\lambda x}\frac{(\lambda
x)^i}{i!}\mathrm{d}B(x).
$$
On the other hand, for the system that is being in state $E_j$, there is
the
probability
$$
\sum_{i=n-j+k+1}^\infty\int_0^\infty\mathrm{e}^{-\lambda x}\frac{(\lambda
x)^i}{i!}\mathrm{d}B(x),
$$
that the losses that occur during a service time, all are
$k$-consecutive. (In this case the number of arrivals during a
service time is not smaller than $n-j+k+1$.)
 Let $q_j=\mathrm{P}(E_j)$ ($q_1+q_2+\ldots+q_n=1$). Then, by
renewal reward theorem \cite{Ross 2000} we have the following
representation:
\begin{equation}\label{CL3}
q_j=\frac{\mathrm{E}T_j-\mathrm{E}T_{j-1}}{\mathrm{E}T_n-\mathrm{E}T_{0}},
\end{equation}
where $\mathrm{E}T_j$, $j=0,1,\ldots,n$, satisfy recurrence relation
\eqref{2.1}. By
using \eqref{CL3}, for the expected number of $k$-consecutive losses
during a busy
period of $M/GI/1/n$ queueing system (which is denoted by
$\mathrm{E}L_{n,k}$) we have:
\begin{equation}\label{CL4}
\mathrm{E}L_{n,k}=c_{n,k}\mathrm{E}L_n,
\end{equation}
where
\begin{equation}\label{CL5}
c_{n,k}=\frac{\sum_{j=1}^n(\mathrm{E}T_j-\mathrm{E}T_{j-1})\sum_{i=n-j+k+1}^\infty\int_0^\infty\mathrm{e}^{-\lambda
x}\frac{(\lambda
x)^i}{i!}\mathrm{d}B(x)}{\sum_{j=1}^n(\mathrm{E}T_j-\mathrm{E}T_{j-1})\sum_{i=n-j+2}^\infty\int_0^\infty\mathrm{e}^{-\lambda
x}\frac{(\lambda x)^i}{i!}\mathrm{d}B(x)}.
\end{equation}
Note, that in the case of the $M/GI/1/0$ queueing system the result is
trivial. We
have:

$$\mathrm{E}L_{0,k}=\rho\left(1-\sum_{i=0}^{k-1}\int_0^\infty\mathrm{e}^{-\lambda
x}\frac{(\lambda x)^i}{i!}\mathrm{d}B(x)\right).$$

The value $c_{n,k}$ in \eqref{CL4} is a bounded constant. It approaches
the limit as
$n\to\infty$. Therefore, the asymptotic order of $\mathrm{E}L_{n,k}$ is
the same as this
of $\mathrm{E}L_{n}$. For example, in the case $\rho=1$,
$\mathrm{E}L_{n,k}\leq1$ is
bounded for all $n$ and therefore converges to the limit as $n\to\infty$.
(According to
the explicit representation \eqref{CL5} there exists the limit of
$c_{n,k}$ as
$n\to\infty$.)

The \textit{open} question concerns the local property of
$\mathrm{E}L_{n,k}$ when the
values $n$ are different and $\rho=1$. The question is: \textit{whether
or not
$\mathrm{E}L_{n,k}$ is the same constant for all $n\geq0$}. The
complicated explicit
formula \eqref{CL5} does not permit us to answer to this question easily.
In all
likelihood, this property does not hold in the general case. However, in
the particular
case of exponentially distributed service times, the answer on this
question is
definite: \textit{$\mathrm{E}L_{n,k}$ is the same for all $n$}. To prove
this, notice
first that when $\rho=1$, we have
$\mathrm{E}T_j-\mathrm{E}T_{j-1}=\frac{1}{\mu}=\frac{1}{\lambda}$ for all
$j=1,2,\ldots,n$. Therefore, by assuming that
$B(x)=1-\mathrm{e}^{-\lambda x}$ we arrive
at the elementary calculations supporting this property. The exact
calculations show
that in this case $c_{n,k}=\left(\frac{1}{2}\right)^{k-1}$. The same
result can be
easily obtained by an alternative way in which, under the assumption
$\rho=1$, the
Markovian queueing system is represented as a symmetric random walk.

Let us now study the asymptotic behaviour of $\mathrm{E}L_{n,k}$. The
only two cases
$\rho=1$ and $\rho>1$ are considered here. (In the case $\rho<1$ we
obtain the trivial
result: $\lim_{n\to\infty}\mathrm{E}L_{n,k}=0$.)

Consider first the case $\rho=1$ and $\int_0^\infty
x^2\mathrm{d}B(x)<\infty$. In this
case according to Theorem \ref{thm3-bp} (which is derivative from Theorem
\ref{thm2Postnikov}), as $j\to\infty$, we have:
\begin{equation}\label{CL6}
\mathrm{E}T_j-\mathrm{E}T_{j-1}=\frac{2}{\lambda\rho_{2}}+o(1),
\end{equation}
where
$$
\rho_2=\lambda^2\int_0^\infty x^2\mathrm{d}B(x).
$$
Therefore, in this case
\begin{equation}\label{CL7}
\begin{aligned}
\lim_{n\to\infty}c_{n,k}&=\lim_{n\to\infty}\frac{\sum_{j=1}^n\sum_{i=n-j+k+1}^\infty\int_0^\infty\mathrm{e}^{-\lambda
x}\frac{(\lambda
x)^i}{i!}\mathrm{d}B(x)}{\sum_{j=1}^n\sum_{i=n-j+2}^\infty\int_0^\infty\mathrm{e}^{-\lambda
x}\frac{(\lambda x)^i}{i!}\mathrm{d}B(x)}\\ &=\frac{\sum_{i=1}^\infty
i\int_0^\infty
\mathrm{e}^{-\lambda x}\frac{(\lambda
x)^{i+k}}{(i+k)!}\mathrm{d}B(x)}{\sum_{i=1}^\infty
i\int_0^\infty \mathrm{e}^{-\lambda x}\frac{(\lambda
x)^{i+1}}{(i+1)!}\mathrm{d}B(x)}\\
 &=\frac{\sum_{i=1}^\infty i\int_0^\infty
\mathrm{e}^{-\lambda x}\frac{(\lambda
x)^{i+k}}{(i+k)!}\mathrm{d}B(x)}{\int_0^\infty\mathrm{e}^{-\lambda
x}\mathrm{d}B(x)}.
\end{aligned}
\end{equation}
Therefore, from the known result $\mathrm{E}L_n=1$ for all $n\geq0$
(Theorem \ref{thm1})
we arrive at
\begin{equation}\label{CL8}
\lim_{n\to\infty}\mathrm{E}L_{n,k}=\frac{\sum_{i=1}^\infty i\int_0^\infty
\mathrm{e}^{-\lambda x}\frac{(\lambda
x)^{i+k}}{(i+k)!}\mathrm{d}B(x)}{\widehat{B}(\lambda)}.
\end{equation}

Let us now consider the case $\rho>1$. In this case for the difference
$\mathrm{E}T_j-\mathrm{E}T_{j-1}$, as $j\to\infty$, according to Theorem
\ref{thm1-bp},
rel. \eqref{ET_n rho geq1}, we have the asymptotic expansion:
\begin{equation}\label{CL9}
\mathrm{E}T_j-\mathrm{E}T_{j-1}=\frac{\rho}{\lambda[1+\lambda\widehat{B}^\prime
(\lambda-\lambda\varphi)]}\frac{\varphi^{j-1}(1-\varphi)}{\varphi^{2j-1}}+o(1).
\end{equation}
Therefore,
\begin{equation}\label{CL10}
\lim_{n\to\infty}c_{n,k}=\lim_{n\to\infty}\frac{\sum_{j=1}^n\varphi^{j-1}\sum_{i=n-j+k+1}^\infty\int_0^\infty\mathrm{e}^{-\lambda
x}\frac{(\lambda
x)^i}{i!}\mathrm{d}B(x)}{\sum_{j=1}^n\varphi^{j-1}\sum_{i=n-j+2}^\infty\int_0^\infty\mathrm{e}^{-\lambda
x}\frac{(\lambda x)^i}{i!}\mathrm{d}B(x)}.
\end{equation}
Thus, we have the following theorem.
\begin{thm}\label{thm-cl} If $\rho<1$, then
$$
\lim_{n\to\infty}\mathrm{E}L_{n,k}=0.
$$
If $\rho=1$ and $\rho_2=\lambda^2\int_0^\infty x^2\mathrm{d}B(x)<\infty$,
then
$$
\lim_{n\to\infty}\mathrm{E}L_{n,k}=\frac{\sum_{i=1}^\infty i\int_0^\infty
\mathrm{e}^{-\lambda x}\frac{(\lambda
x)^{i+k}}{(i+k)!}\mathrm{d}B(x)}{\widehat{B}(\lambda)}.
$$
If $\rho>1$, then
$$
\lim_{n\to\infty}\frac{\mathrm{E}L_{n,k}}{\mathrm{E}L_{n}}=\lim_{n\to\infty}\frac{\sum_{j=1}^n\varphi^{j-1}\sum_{i=n-j+k+1}^\infty\int_0^\infty\mathrm{e}^{-\lambda
x}\frac{(\lambda
x)^i}{i!}\mathrm{d}B(x)}{\sum_{j=1}^n\varphi^{j-1}\sum_{i=n-j+2}^\infty\int_0^\infty\mathrm{e}^{-\lambda
x}\frac{(\lambda x)^i}{i!}\mathrm{d}B(x)},
$$
where $\varphi$ is the least positive root of equation \eqref{2.10}, and
$\mathrm{E}L_{n}$ is defined by asymptotic expansion \eqref{2.9} of
Theorem \ref{thm1}.
\end{thm}

\section{Losses in the $GI/M/1/n$ queue}\label{GIM1n queue}

Loss probabilities in $M^X/GI/1/n$ and $GI/M^Y/1/n$ queues have been
studied by Miyazawa
\cite{Miyazawa 1990}. In the special case of the $GI/M/1/n$ queue, from
the results of
Miyazawa \cite{Miyazawa 1990} we have as follows.

Let $\mu$ denote the parameter of service time distribution, let $A(x)$
be probability
distribution function of interarrival time, and let $\lambda$ be the
reciprocal of the
expected interarrival time. Denote the load of the system by
$\rho=\frac{\lambda}{\mu}$
and let $\rho_j=\int_0^\infty(\mu x)^j\mathrm{d}A(x)$. (Notice that
$\rho_1=\frac{1}{\rho}$.)

Miyazawa \cite{Miyazawa 1990} proved that the loss probability
$P_{loss}=P_{loss}(n)$
does exist for any $\rho$, and\footnote{Dependence on parameter $n$ will
be omitted. It
is only indicated once in the formulation of Theorem \ref{thm6}.}
\begin{equation}
\label{3.1} P_{loss}=\frac{1}{\sum_{j=0}^n r_j},
\end{equation}
where the generating function $R(z)$ of $r_j$, $j=0,1,\ldots$, is as
follows:
\begin{equation}
\label{3.2} R(z)=\sum_{j=0}^\infty r_jz^j=\frac{(1-z)\widehat{A}(\mu-\mu
z)}{\widehat{A}(\mu-\mu z)-z}, \ \ |z|<\varphi,
\end{equation}
$\widehat{A}(s)=\int_0^\infty\mathrm{e}^{-sx}\mathrm{d}A(x)$, and
$\varphi$ is the least
positive root of the equation
\begin{equation}
\label{3.3} z = \widehat{A}(\mu-\mu z).
\end{equation}
(For the least positive root of the equation we use the same notation
$\varphi$ as for
the least positive root of similar equation \eqref{2.10}. We hope, that
it does not
confuse the readers of this survey.)

From \eqref{3.2} we have:
\begin{equation}\label{3.4}
\begin{aligned}
R(z)=\frac{(1-z)\widehat{A}(\mu-\mu z)}{\widehat{A}(\mu-\mu z)-z}
&=\frac{\widehat{A}(\mu-\mu z)}{\widehat{A}(\mu-\mu
z)-z}-z\frac{\widehat{A}(\mu-\mu
z)}{\widehat{A}(\mu-\mu z)-z}\\ &=\widetilde{R}(z)-z\widetilde{R}(z),
\end{aligned}
\end{equation}
where
\begin{equation*}
\label{3.5}
\widetilde{R}(z)=\sum_{j=0}^\infty\widetilde{r}_jz^j=\frac{\widehat{A}(\mu-\mu
z)}{\widehat{A}(\mu-\mu z)-z}.
\end{equation*}
Note that
\begin{eqnarray}\label{3.6}
r_0&=&\widetilde{r}_0,\nonumber\\
r_{j+1}&=&\widetilde{r}_{j+1}-\widetilde{r}_j, \ \
j\geq0.\nonumber
\end{eqnarray}
Therefore,
$$
\sum_{j=0}^n r_j=\widetilde{r}_n
$$
and therefore from \eqref{3.1} we have
\begin{equation}
\label{3.8} P_{loss}=\frac{1}{\widetilde{r}_n}.
\end{equation}
The sequence $\widetilde{r}_n$ satisfies the convolution type recurrence
relation of
\eqref{1.1}. Specifically,
\begin{equation}\label{3.9}
\widetilde{r}_n=\sum_{j=0}^n
\widetilde{r}_{n-j+1}\int_0^\infty\mathrm{e}^{-\mu
x}\frac{(\mu x)^j}{j!}\mathrm{d}A(x).
\end{equation}
Therefore by applying Theorem \ref{thmTakacs} we arrive at the following
theorem (for
details of the proof see \cite{Abramov 2002}).

\begin{thm}
\label{thm4} If $\rho<1$, then as $n\to\infty$
\begin{equation}
\label{3.10}
P_{loss}=\frac{(1-\rho)[1+\mu\widehat{A}^\prime(\mu-\mu\varphi)]\varphi^n}
{1-\rho-\rho[1+\mu\widehat{A}^\prime(\mu-\mu\varphi)]\varphi^n}+o(\varphi^{2n}).
\end{equation}
If $\rho=1$ and $\rho_2<\infty$, then
\begin{equation}
\label{3.11} \lim_{n\to\infty}nP_{loss}=\frac{\rho_2}{2}.
\end{equation}
If $\rho>1$, then
\begin{equation}
\lim_{n\to\infty}P_{loss}=\frac{\rho-1}{\rho}.
\end{equation}
\end{thm}

Using Tauberian Theorems \ref{thm1Postnikov} and \ref{thm2Postnikov} one
can improve
\eqref{3.11}. In this case we have the following two theorems.

\begin{thm}
\label{thm5} Assume that $\rho=1$ and $\rho_3<\infty$. Then, as
$n\to\infty$,
\begin{equation}
\label{3.12} P_{loss}=\frac{\rho_2}{2n}+O\left(\frac{\log n}{n^2}\right).
\end{equation}
\end{thm}

The proof of this theorem follows immediately from Theorem
\ref{thm1Postnikov}.

\begin{thm}
\label{thm6} Assume that $\rho=1$ and $\rho_2<\infty$. Then, as
$n\to\infty$,
\begin{equation}
\frac{1}{P_{loss}(n+1)}-\frac{1}{P_{loss}(n)}=\frac{2}{\rho_2}+o(1).
\end{equation}
\end{thm}

\begin{proof}
The proof of this theorem is similar to that of the proof of Theorem
\ref{thm3-bp}.
\end{proof}

Choi and Kim \cite{Choi and Kim 2000} and Choi, Kim and Wee \cite{Choi
Kim and Wee 2000}
study the asymptotic behavior of the stationary probabilities and loss
probabilities.
Some of these results are close to the results obtained in Theorems
\ref{thm4} and
\ref{thm6}. For more details see the discussion section of \cite{Abramov
2002}.

The further heavy traffic analysis leads to the following two theorems.

\begin{thm}
\label{thm7} Let $\rho=1-\delta$, $\delta>0$, and let $\delta n\to C>0$
as $n\to\infty$
and $\delta\to0$. Assume that $\rho_3(n)$ is a bounded sequence, and
there exists
$\widetilde\rho_2=\lim_{n\to\infty}\rho_2(n)$. Then,
\begin{equation}\label{3.18}
P_{loss}=\frac{\delta\mathrm{e}^{-2C/\widetilde{\rho}_2}}
{1-\mathrm{e}^{-2C/\widetilde{\rho}_2}}[1+o(1)].
\end{equation}
\end{thm}

\begin{thm}
\label{thm8} Let $\rho=1-\delta$, $\delta>0$, and let $\delta n\to 0$ as
$n\to\infty$
and $\delta\to0$. Assume that $\rho_3(n)$ is a bounded sequence, and
there exists
$\widetilde\rho_2=\lim_{n\to\infty}\rho_2(n)$. Then,
\begin{equation}\label{3.19}
P_{loss}=\frac{\widetilde{\rho}_2}{2n}+o\left(\frac{1}{n}\right).
\end{equation}
\end{thm}

The proof of the both of these theorems is based on asymptotic
expansions, which are
analogous to those of \eqref{2.12} and \eqref{2.13}.

Similar results, related to the heavy traffic analysis of $GI/M/1/n$
queues, have been
obtained in \cite{Whitt 2004}.

\section{Losses in the $GI/M/m/n$ queue}\label{GIMmn queue}

Tak\'acs' theorem has also been applied for analysis of the loss
probabilities in
multiserver $GI/M/m/n$ queueing systems \cite{Abramov 2007-1}. An
application of
Tak\'acs' theorem in this case, however, is not entirely straightforward
and based on
special approximations. Specifically, the recurrence relation of
convolution type
\eqref{1.1} is valid only in limit, and a technically hard analytic proof
with
complicated notation is required in order to reduce the equations
describing the
stationary loss probability to the asymptotic recurrence relation of
convolution type
\eqref{1.1}.

In this section we do not present the details of the proofs. These
details of the proofs
can be found in \cite{Abramov 2007-1}. We only formulate the theorems and
discuss
general features and differences between the $GI/M/m/n$ and $GI/M/1/n$
cases in the
corresponding theorems.

As in the case $m=1$ above, $A(x)$ is the probability distribution
function of
interarrival time, $\lambda$ is the reciprocal of the expected
interarrival time,
$\widehat{A}(s)$ is the Laplace-Stieltjes transform of $A(x)$, where the
argument $s$ is
assumed to be nonnegative. The parameter of the service time distribution
is denoted by
$\mu$, and the load of the system $\rho=\lambda/(m\mu)$. The least
positive root of the
equation $z=\widehat{A}(m\mu-m\mu z)$ will be denoted $\varphi_m$. The
loss probability
$P_{loss}$ is now dependent of $m$ and $n$.

The following theorem on the stationary loss probability have been
established in
Abramov \cite{Abramov 2007-1}.

\begin{thm}\label{thm9}
If $\rho>1$, then for any $m\geq1$,
\begin{equation}
\label{4.1} \lim_{n\to\infty}P_{loss}=\frac{\rho-1}{\rho}.
\end{equation}
If $\rho=1$ and $\rho_2=\int_0^\infty(\mu x)^2\mathrm{d}A(x)<\infty$,
then for any
$m\geq1$,
\begin{equation}\label{4.2}
\lim_{n\to\infty}nP_{loss}=\frac{\rho_2}{2}.
\end{equation}
If $\rho=1$ and $\rho_3=\int_0^\infty(mx)^3\mathrm{d}A(x)<\infty$, then
for large $n$
and any $m\geq1$,
\begin{equation}\label{4.3}
P_{loss}=\frac{\rho_2}{2n}+O\left(\frac{\log n}{n^2} \right).
\end{equation}
If $\rho<1$, then for $P_{loss}$ we have the limiting
relation:~\footnote{The asymptotic
relation, which is presented here, is more exact than that was presented
in
\cite{Abramov 2007-1} in the formulation of Theorem 3.1 of that paper.}
\begin{equation}
\label{4.4} P_{loss}=
K_m\frac{(1-\rho)[1+m\mu\widehat{A}^\prime(m\mu-m\mu\varphi_m)]\varphi_m^{n-1}}
{1-\rho-\rho[1+m\mu\widehat{A}^\prime(m\mu-m\mu\varphi_m)]\varphi_m^{n-1}}+o(\varphi_m^{2n}),
\end{equation}
where
\begin{equation}\label{4.5}
K_m=\left[1+(1-\varphi_m)\sum_{j=1}^m\binom{m}{j}\frac{C_j}{1-\sigma_j}\cdot
\frac{m(1-\sigma_j)-j}{m(1-\varphi_m)-j}\right]^{-1},
\end{equation}
\begin{equation*}\sigma_j=\int_0^\infty\mathrm{e}^{-j\mu
x}\mathrm{d}A(x),\end{equation*}
\begin{equation*}
C_j=\prod_{i=1}^j\frac{1-\sigma_j}{\sigma_j}.
\end{equation*}
\end{thm}

Relations \eqref{4.1}, \eqref{4.2} and \eqref{4.4} follow from
Tak\'acs' theorem \ref{thmTakacs}, relation \eqref{4.3} follows
from Tauberian theorem \ref{thm1Postnikov}. Relations \eqref{4.1},
\eqref{4.2} and \eqref{4.3} are as the corresponding results for
the $GI/M/1/n$ queue. However, there is the difference between
\eqref{4.4} and \eqref{3.10} associated with special
representation \eqref{4.5} of the coefficient $K_m$, so the
analysis of the case $\rho<1$ is more delicate than that in the
cases $\rho>1$ and $\rho=1$.

By alternative methods, the asymptotic loss and stationary
probabilities in $GI/M/m/n$ queues have been studied in several
papers. The analytic proofs given in these papers are much more
difficult than those by application of Tak\'acs' theorem
\ref{thmTakacs}. We refer Kim and Choi \cite{Kim and Choi 2003},
Choi et al. \cite{Choi et al 2003} and Simonot \cite{Simonot
1998}, where the readers can find the proofs by using the standard
analytic techniques.

It is surprising that in the heavy traffic case, where $\rho$ approaches
1 from the
below, we obtain the same asymptotic representation for the loss
probability as in the
$GI/M/1/n$ case. Although the expression for $K_m$ \eqref{4.5} looks
complicated, its
asymptotic expansion when $\rho=1-\delta$, $\delta>0$, and $\delta n\to
C$ as
$n\to\infty$ and $\delta\to0$ is very simple:
\begin{equation*}
K_m=1+O(\delta).
\end{equation*}
Therefore, we arrive at the following theorem.

\begin{thm}
\label{thm10} Let $\rho=1-\delta$, $\delta>0$, and let $\delta n\to C$ as
$n\to\infty$
and $\delta\to0$. Suppose that $\rho_3=\rho_3(n)$ is a bounded sequence,
and there
exists $\widetilde{\rho_2}=\lim_{n\to\infty}\rho_2(n)$. Then, in the case
$C>0$ for any
$m\geq1$ we have
\begin{equation}\label{4.8}
P_{loss}=\frac{\delta\mathrm{e}^{-2C/\widetilde{\rho}_2}}{1-
\mathrm{e}^{-2C/\widetilde{\rho}_2}}[1+o(1)].
\end{equation}
In the case $C=0$ for any $m\geq1$ we have
\begin{equation}\label{4.9}
P_{loss}=\frac{\widetilde{\rho}_2}{2n}+o\left(\frac{1}{n}\right).
\end{equation}
\end{thm}

It is readily seen that the asymptotic representation given by
\eqref{4.8} and
\eqref{4.9} are exactly the same as the corresponding asymptotic
representations of
\eqref{3.18} and \eqref{3.19}.

\section{Future research problems}

In the previous sections, applications of Tak\'acs' theorem to
queueing systems have been discussed. In this section, which
concludes Part 1, we formulate problems for the future
applications of Tak\'acs' theorem to different queueing systems.
Specifically, we discuss possible ways of application of Tak\'acs
theorem to retrial queueing systems.

Consider a single server queueing system with Poisson input rate
$\lambda$ and $n$ number of waiting places. A customer, who upon
arrival finds all waiting places busy, goes to the secondary
system, and after some  random time
arrives at the main system again. If at least one of waiting
places is free, the customer joins the main queue. Assuming that
service times in the main system are independent and identically
distributed random variables, one can interest in asymptotic
behavior of the fraction of retrials with respect to the total
number of arrivals during a busy period.

Assumptions on retrials
can be very different. One of the simplest models is based on the
assumption that times between retrials are independent identically
distributed random variables, and their distribution is
independent of the number of customer in the secondary system.
In another model times between retrials are
exponentially distributed with parameter depending on the number
of customers in the secondary queue. The most familiar model
amongst them is a model with linear retrial policy: If the number
of customers in the secondary system is $j$, then the retrial rate
is $j\mu$.

Asymptotic analysis as $n\to\infty$ of the fraction of retrials
with respect to the total number of customers arrived during a
busy period is a significant performance characteristic of these
large retrial queueing systems. An application of Tak\'{a}cs'
theorem should be based on comparison of the desired
characteristics of retrial queueing systems with the corresponding
characteristics of the $M/GI/1/n$ queueing system with losses.
Specifically, denoting by $L_n$ the number of losses in the
$M/GI/1/n$ queueing system and by $R_n$ the number of retrials in
a given queueing system with retrials, we interest in finding the
values $c_n$ supporting the equality
$\mathrm{E}R_n=c_n\mathrm{E}L_n$. If the properties of the
sequence $c_n$ are well-specified (for example, there exists a
limit or an appropriate estimate), then the application of
Tak\'acs' theorem to $\mathrm{E}R_n$ as $n\to\infty$ becomes
elementary. (The asymptotic behavior of $\mathrm{E}L_n$ is given
by Theorem \ref{thm1}.)

\part{Applications to stochastic models of communication systems and
dam/storage systems}

\section{Asymptotic analysis of the number of lost messages in
communication systems}\label{Lost messages} In this section we study
losses in optical
telecommunication networks, where we develop the results on losses in
$M/GI/1/n$ queues
considered in Section 1. We discuss the results established in
\cite{Abramov 2004}.

Long messages being sent are divided into a random number of packets
which are
transmitted independently of one another. An error in transmission of a
packet results
in a loss of the entire message. Messages arrive to the $M/GI/1$ finite
buffer model
(the exact description of the model is given below) and can be lost in
two cases as
either at least one of its packets is corrupted or the buffer is
overflowed.

The model is the following extension of the usual $M/GI/1/n$ system. We
consider
queueing system with Poisson input rate $\lambda$ of batch arrivals. The
system serves
each of these batches, and each service time has probability distribution
$B(x)$ with
mean $\frac{1}{\mu}$. The random batches $\kappa_i$ are bounded from the
above and
below, so that
\begin{equation}
\label{5.1} \mathrm{P}\{\kappa^{lower}\leq\kappa_i\leq\kappa^{upper}\}=1.
\end{equation}
$\kappa_i$ is the number of packets associated with the $i$th message.

 Let
us denote
\begin{equation*}
\zeta=\sup\left\{m: \sum_{i=1}^m\kappa_i\leq N\right\},
\end{equation*}
and according to assumption \eqref{5.1} there are two fixed values
$\zeta^{lower}$ and
$\zeta^{upper}$ depending on $N$ and
$\mathrm{P}\{\zeta^{lower}\leq\zeta\leq\zeta^{upper}\}=1$. Let $\xi_i$ be
the number of
messages in the queue immediately before arrival of the $i$th message.
Then the message
is lost if $\xi_i>\zeta_i$. Otherwise it joins the queue. $\zeta_i$ is
the $i$th
(generic) random level in terms of a number of possible messages in the
system. The
special case when $\mathrm{P}\{\kappa_i=l\}=1$ (a message contains a
fixed (non-random)
number of packets) leads to the standard $M/GI/1/n$ queueing system,
where
$n=\left\lfloor \frac{N}{l}\right\rfloor$ is the integer part of
$\frac{N}{l}$.

It is also assumed that each message is marked with probability $p$, and
we study the
asymptotic behavior of the loss probability under assumptions that
$\mathrm{E}\zeta$
increases to infinity and $p$ vanishes. The lost probability is the
probability that the
message is either marked or lost because of overflowing the queue. We
demonstrate an
application of Tak\'acs' theorem as well as Theorems \ref{thm1Postnikov}
and
\ref{thm2Postnikov} for the solution of all of these problems.

The queueing system described above is not standard, and the explicit
representation for
its characteristics can not be obtained traditionally. We will introduce
a class
$\Sigma$ of queueing systems. A simple representative of this class is
the system
$\mathcal{S}_1$ which has been determined above for which we will
establish the balance
equation for the expectations of accepted/rejected customers in the
system which are
similar to those of \eqref{2.2} and \eqref{2.3} in the case of the
standard $M/GI/1/n$
system. Such equations will be written explicitly in the sequel.

In order to define the class $\Sigma$ let us first study elementary
processes of the
system $\mathcal{S}_1$.

Let $\xi_i$ denote the number of messages in the system $\mathcal{S}_1$
immediately
before arrival of the $i$th message, $\xi_1=0$, and let $s_i$ denote the
number of
service completions between the $i$th and $i+1$st arrivals. Clearly, that
\begin{equation}
\label{5.3} \xi_{i+1}=\xi_i-s_i+\mathrm{1}_{\{\xi_i\leq\zeta_i\}},
\end{equation}
where the term $\mathrm{1}_{\{\xi_i\leq\zeta_i\}}$ in \eqref{5.3}
indicates that the
$i$th message is accepted. Obviously, that $s_i$ is not greater than
$\xi_i+\mathrm{1}_{\{\xi_i\leq\zeta_i\}}$.

Let us consider now a new queueing system as above with the same rate of
Poisson input
$\lambda$ and the same probability distribution function of the service
time $B(x)$, but
with another sequence $\widetilde{\zeta}_1$, $\widetilde{\zeta}_2$,\ldots
of
\textit{arbitrary dependent} sequence of random variables all having the
same
distribution as $\zeta$. Let $\widetilde{\xi}_i$ denote the number of
messages
immediately before arrival of the $i$th message ($\widetilde{\xi}_1=0)$,
and let
$\widetilde{s}_i$ denote the number of service completions between the
$i$th and $i+1$st
arrival. Analogously to \eqref{5.3} we have
\begin{equation}
\label{5.4} \widetilde\xi_{i+1}=\widetilde\xi_i-\widetilde{s}_i+
\mathrm{1}_{\{\widetilde\xi_i\leq\widetilde\zeta_i\}}.
\end{equation}

\begin{defn}\label{defn5.1}
The queueing system $\mathcal{S}$ is said to belong to the set $\Sigma$
of queueing
systems if $\mathrm{E}\widetilde\xi_i=\mathrm{E}\xi_i$,
$\mathrm{E}\widetilde{s}_i=\mathrm{E}s_i$ and

$\mathrm{P}\{\widetilde\xi_i\leq\widetilde\zeta_i\}=\mathrm{P}\{\xi_i\leq\zeta_i\}$
for
all $i\geq1$.
\end{defn}

Let us now consider an example of queueing system belonging to the set
$\Sigma$. The
example is $\widetilde\zeta_1=\widetilde\zeta_2=\ldots$. Denote this
queueing system by
$\mathcal{S}_2$. This example is artificial, but it helps to easily study
this specific
system, and together with this system all of the systems belonging to
this class
$\Sigma$.

Specifically, for this system $\mathcal{S}_2$ according to the induction
for all
$i\geq1$ we have:
\begin{equation}
\label{5.5} \mathrm{E}\widetilde{s}_i=\mathrm{E}s_i,
\end{equation}
\begin{equation}
\label{5.6} \mathrm{P}\{\widetilde\xi_i\leq\widetilde\zeta_i\}=
\mathrm{P}\{\xi_i\leq\zeta_i\},
\end{equation}
and
\begin{equation}
\label{5.7} \mathrm{E}\widetilde\xi_i-\mathrm{E}\widetilde{s}_i+
\mathrm{P}\{\widetilde\xi_i\leq\widetilde\zeta_i\}
=\mathrm{E}\xi_i-\mathrm{E}s_i+\mathrm{P}\{\xi_i\leq\zeta_i\}.
\end{equation}
Relations \eqref{5.5}, \eqref{5.6} and \eqref{5.7} enable us to conclude
that
$\mathcal{S}_2\in\Sigma$. Furthermore, from these relations \eqref{5.5},
\eqref{5.6} and
\eqref{5.7} all of characteristics of all of queueing systems from the
class $\Sigma$ do
exist and are the same. Therefore, for our conclusion it is enough to
study the queueing
system $\mathcal{S}_2$, which is the simplest than all other.

Let $\widetilde{T}_{\zeta}$ denote a busy period of system
$\mathcal{S}_2$. By the
formula for the total expectation
\begin{equation}
\label{5.8}
\mathrm{E}\widetilde{T}_\zeta=\sum_{i=\zeta^{lower}}^{\zeta^{upper}}\mathrm{E}T_i
\mathrm{P}\{\zeta=i\},
\end{equation}
where $\mathrm{E}T_i$ is the expectation of the busy period of an
$M/GI/1/i$ queueing
system with the same interarrival and service time distributions. The
expectations
$\mathrm{E}T_i$ are determined from the convolution type recurrence
relations of
\eqref{2.1}, and all of the results of the above theory related to
$M/GI/1/n$ can be
applied here. Then one can write
$\mathrm{E}T_\zeta=\mathrm{E}\widetilde{T}_\zeta$, and
therefore for the queueing system $\mathcal{S}_1$ we have the same
relation as
\eqref{5.8}:
\begin{equation}
\label{5.9}
\mathrm{E}{T}_\zeta=\sum_{i=\zeta^{lower}}^{\zeta^{upper}}\mathrm{E}T_i
\mathrm{P}\{\zeta=i\},
\end{equation}

Along with the notation $T_\zeta$ for the busy period of the queueing
system
$\mathcal{S}_1$ we consider also the following characteristics of this
queueing system.
Let $I_\zeta$ denote an idle period, and let $P_\zeta$, $M_\zeta$ and
$R_\zeta$ denote
the number of processed messages, the number of marked messages and the
number of
refused messages respectively. We will use the following terminology. The
term
\textit{refused} message is used for the case of overflowing the buffer,
while the term
\textit{lost} message is used for the case where a message is either
refused or marked.
The number of lost messages during a busy period is denoted by $L_\zeta$.
Analogously,
by lost probability we mean the probability when the arrival message is
lost.

\begin{lem}
\label{lem1} For the expectations $\mathrm{E}T_\zeta$,
$\mathrm{E}P_\zeta$,
$\mathrm{E}M_\zeta$, $\mathrm{E}R_\zeta$ we have the following
representations:
\begin{equation}
\label{5.10} \mathrm{E}P_\zeta=\mu\mathrm{E}T_\zeta,
\end{equation}
\begin{equation}
\label{5.11} \mathrm{E}M_\zeta=p\mathrm{E}P_\zeta,
\end{equation}
\begin{equation}
\label{5.12} \mathrm{E}R_\zeta=(\rho-1)\mathrm{E}P_\zeta+1.
\end{equation}
\end{lem}

The proof of this lemma is based on Wald's equations \cite{Feller
1966}, p.384 and similar to the proof of relation \eqref{2.5} from
the system of equations \eqref{2.2} and \eqref{2.3}. The relations
\eqref{5.9}, \eqref{5.10}, \eqref{5.11} and \eqref{5.12} all
together define all of the required expectations, and all of them
are expressed via recurrence relation of convolution type.
Therefore, one can apply Tak\'acs' theorem and Tauberian theorems
\ref{thm1Postnikov} and \ref{thm2Postnikov}.

 Denote: $\rho_j=\int_{0}^\infty(\lambda
x)^j\mathrm{d}B(x)$, $j=1,2,\ldots$, where according to the earlier
notation
$\rho=\rho_1$ is the load parameter of the system.

We write $\zeta=\zeta(N)$ to point out the dependence on parameter $N$.
As $N$ tends to
infinity, both $\zeta^{lower}$ and $\zeta^{upper}$ tend to infinity, and
together with
them $\zeta(N)$ tends to infinity almost surely (a.s.) For the above
characteristics of
the queueing system we have the following theorems.

\begin{thm}
\label{thm11} If $\rho<1$, then
\begin{equation}
\label{5.13} \lim_{N\to\infty}\mathrm{E}P_{\zeta(N)}=\frac{1}{1-\rho}.
\end{equation}
If $\rho=1$ and $\rho_2<\infty$, then
\begin{equation}
\label{5.14}
\lim_{N\to\infty}\frac{\mathrm{E}P_{\zeta(N)}}{\mathrm{E}\zeta(N)}=\frac{2}{\rho_2}.
\end{equation}
If $\rho>1$, then
\begin{equation}
\label{5.15} \lim_{N\to\infty}\left[\mathrm{E}P_{\zeta(N)}-
\frac{1}{\mathrm{E}\varphi^{\zeta(N)}
[1+\lambda\widehat{B}^\prime(\lambda-\lambda\varphi)]}\right]
=\frac{1}{1-\rho},
\end{equation}
where $\widehat{B}(s)$ is the Laplace-Stieltjes transform of the
probability
distribution function $B(x)$ and $\varphi$ is the least positive root of
equation
$z=\widehat{B}(\lambda-\lambda z)$. (See rel. \eqref{2.10}.)
\end{thm}

\begin{thm}
\label{thm12} If $\rho=1$ and $\rho_3<\infty$, then
\begin{equation}
\label{5.16} \mathrm{E}P_{\zeta(N)}=\frac{2}{\rho_2}+O(\log N).
\end{equation}
\end{thm}

\begin{thm}
\label{thm13} If $\rho<1$, then
\begin{equation}
\label{5.17} \lim_{N\to\infty}\mathrm{E}R_{\zeta(N)}=0.
\end{equation}
If $\rho=1$, then for all $N\geq0$
\begin{equation}
\label{5.18} \mathrm{E}R_{\zeta(N)}=1.
\end{equation}
If $\rho>1$, then
\begin{equation}
\label{5.19} \lim_{N\to\infty}\left[\mathrm{E}R_{\zeta(N)}-
\frac{\rho-1}{\mathrm{E}\varphi^{\zeta(N)}
[1+\lambda\widehat{B}^\prime(\lambda-\lambda\varphi)]}\right]=0.
\end{equation}
\end{thm}

Theorems \ref{thm11} and \ref{thm12} follow from the corresponding
Tak\'acs' theorem and Postnikov's Tauberian theorem
\ref{thm1Postnikov}. Theorem \ref{thm13} follows from Tak\'acs'
theorem  and is an analogue of Theorem \ref{thm1} on losses in
$M/GI/1/n$ queues. Relation \eqref{5.18} says that the remarkable
property of losses under the condition $\rho=1$ remains the same
as in the case of the standard $M/GI/1/n$ queueing system. The
details of proofs for all of these theorems as well as the
following theorems of this section can be found in \cite{Abramov
2004}.

In the case where the number of packets in each message is considered to
be fixed, then
from Postnikov's Tauberian Theorem \ref{thm2Postnikov} we have as
follows.

\begin{thm}
\label{thm14} If $\rho=1$ and $\rho_2<\infty$, then as $n\to\infty$,
\begin{equation}
\label{5.20} \mathrm{E}P_{n+1}-\mathrm{E}P_n=\frac{2}{\rho_2}+o(1),
\end{equation}
where the index $n+1$ says that $P_{n+1}$ is the number of processed
messages during a
busy period of the $M/GI/1/n+1$ queueing system.
\end{thm}

Note, that the proof of Theorem \ref{thm14} is based on application of
Tauberian Theorem
\ref{thm2Postnikov} and is exactly the same as the proof of Theorem
\ref{thm6}.

 Under heavy traffic conditions we have as follows.
 \begin{thm}
 \label{thm15} Let $\rho=1+\delta$, $\delta>0$, and $\delta\zeta(N)\to
 C>0$ a.s. as $N\to\infty$ and $\delta\to0$. Assume also that
 $\rho_3=\rho_3(N)$ is a bounded sequence, and there exists $\widetilde
 \rho_2=\lim_{N\to\infty}\rho_2(N)$. Then,
 \begin{equation}\label{5.21}
\mathrm{E}P_{\zeta(N)}=\frac{\mathrm{e}^{2C/\widetilde{\rho}_2}-1}{\delta}+O(1),
 \end{equation}
 \begin{equation}
 \label{5.22}
 \mathrm{E}R_{\zeta(N)}=\mathrm{e}^{2C/\widetilde{\rho}_2}+o(1).
 \end{equation}
 \end{thm}

\begin{thm}
\label{thm16} Let $\rho=1+\delta$, $\delta>0$ and $\delta\zeta(N)\to0$
a.s. as
$N\to\infty$ and $\delta\to0$. Assume also that $\rho_3=\rho_3(N)$ is a
bounded
sequence, and there exists $\widetilde
 \rho_2=\lim_{N\to\infty}\rho_2(N)$. Then,
\begin{equation}
\label{5.23}
\mathrm{E}P_{\zeta(N)}=\frac{2}{\widetilde{\rho}_2}\mathrm{E}\zeta(N)+O(1),
\end{equation}
\begin{equation}
\label{5.24}\mathrm{E}R_{\zeta(N)}= 1+o(1).
\end{equation}
\end{thm}

Asymptotic behavior of the loss probability can be deduced from the above
asymptotic
theorems by using renewal arguments. According to renewal arguments, the
loss
probability is
\begin{equation}
\label{5.25}
\Pi_\zeta=\frac{\mathrm{E}L_\zeta}{\mathrm{E}R_\zeta+\mathrm{E}P_\zeta}
=\frac{\mathrm{E}R_\zeta+\mathrm{E}M_\zeta}{\mathrm{E}R_\zeta+\mathrm{E}P_\zeta}
=\frac{\mathrm{E}R_\zeta+p\mathrm{E}P_\zeta}{\mathrm{E}R_\zeta+\mathrm{E}P_\zeta},
\end{equation}
where $p$ is the probability of losing a message, because one of its
packets is
corrupted.

We have as follows.

\begin{thm}
\label{thm17} If $\rho<1$, then
\begin{equation}\label{5.26}
\lim_{N\to\infty}\Pi_{\zeta(N)}=p.
\end{equation}
Limiting relation \eqref{5.26} is also valid when $\rho=1$ and
$\rho_2<\infty$. If
$\rho>1$, then
\begin{equation}
\label{5.27} \Pi_{\zeta(N)}=\frac{p+\rho-1}{\rho}\cdot
\frac{(\rho-1)+p[1+\lambda\widehat{B}^\prime(\lambda-\lambda\varphi)]\mathrm{E}\varphi^{\zeta(N)}}
{(\rho-1)+[1+\lambda\widehat{B}^\prime(\lambda-\lambda\varphi)]\mathrm{E}\varphi^{\zeta(N)}}
+o(\mathrm{E}\varphi^{\zeta(N)}).
\end{equation}
\end{thm}

\begin{thm}
\label{thm18} If $\rho=1$ and $\rho_3<\infty$, then as $N\to\infty$,
\begin{equation}
\label{5.28} \Pi_{\zeta(N)}=
p+\frac{(1-p)\rho_2}{2\mathrm{E}\zeta(N)}+O\left(\frac{\log
N}{N^2}\right).
\end{equation}
\end{thm}
If under the assumptions of Theorem \ref{thm17} to assume additionally
that $p\to0$,
then in the case $pN\to C>0$ as $p\to0$ and $N\to\infty$ we have:
\begin{equation*}\label{5.29}
\Pi_{\zeta(N)}=\frac{C}{N}+\frac{\rho_2}{2\mathrm{E}\zeta(N)}+O\left(\frac{\log
N}{N^2}\right).
\end{equation*}
In the case $pN\to 0$ as $p\to0$ and $N\to\infty$ we have:
\begin{equation*}
\label{5.30}
\Pi_{\zeta(N)}=\frac{\rho_2}{2\mathrm{E}\zeta(N)}+O\left(p+\frac{\log
N}{N^2}\right).
\end{equation*}

\begin{thm}
\label{thm19} Let $\rho=1+\delta$, $\delta>0$, and $\delta\zeta(N)\to
C>0$ as
$N\to\infty$ and $\delta>0$, and let $p\to0$. Assume also that
$\rho_3=\rho_3(N)$ is a
bounded sequence, and there exists
$\widetilde{\rho}_2=\lim_{N\to\infty}\rho_2(N)$.

(i) If $\frac{p}{\delta}\to D\geq0$, then
\begin{equation}
\label{5.31}
\Pi_{\zeta(N)}=\left(D+\frac{\mathrm{e}^{2C/\widetilde{\rho}_2}}
{\mathrm{e}^{2C/\widetilde{\rho}_2}-1}\right)\delta+o(\delta).
\end{equation}

(ii) If $\frac{p}{\delta}\to\infty$, then
\begin{equation}
\label{5.32} \Pi_{\zeta(N)}= p+O(\delta).
\end{equation}
\end{thm}

\begin{thm}
\label{thm20} Let $\rho=1+\delta$, $\delta>0$, and $\delta\zeta(N)\to0$
as $N\to\infty$
and $\delta>0$, and let $p\to0$. Assume also that $\rho_3=\rho_3(N)$ is a
bounded
sequence, and there exists
$\widetilde{\rho}_2=\lim_{N\to\infty}\rho_2(N)$.

(i) If $\frac{p}{\delta}\to D\geq0$, then
\begin{equation}
\label{5.33} \Pi_{\zeta(N)}= p+\frac{\widetilde{\rho}_2}
{2\mathrm{E}\zeta(N)}+o\left(\frac{1}{N}\right).
\end{equation}

(ii) If $\frac{p}{\delta}\to\infty$, then we have \eqref{5.32}.
\end{thm}

In the special case where each message contains exactly $l$ packets,
$n=\lfloor
\frac{N}{l}\rfloor$, we have as follows.

\begin{thm}
\label{thm21} If $\rho_1=1$ and $\rho_2<\infty$, then as $n\to\infty$
\begin{equation*}
\label{5.34}
\Pi_{n+1}-\Pi_n=\frac{\frac{1}{n(n+1)}~\frac{2}{\rho_2}(p-1)}
{\left(\frac{2}{\rho_2}+\frac{1}{n+1}\right)\left(\frac{2}{\rho_2}+\frac{1}{n}\right)}
+o\left(\frac{1}{n^2}\right).
\end{equation*}
\end{thm}

Asymptotic Theorems \ref{thm17} -- \ref{thm20} on loss probability enable
us to make
conclusion on adding redundant packets into messages. The standard
assumption given in
\cite{Abramov 2004} is that a redundant packet decreases probability $p$
by $\gamma$
times but increases the load of the system by $\widetilde\gamma$ times.
The analysis in
\cite{Abramov 2004} showed that in the case $\rho<1$ the adding a number
of redundant
packets can decrease the loss probability with geometric rate while
$\rho\leq1$. It was
also shown that in critical cases where $\rho=1+\delta$ adding redundancy
can be
profitable as well. The details of the analysis can be found in
\cite{Abramov 2004}.

There is a large number of papers in the literature on communication
systems and
networks studying related problems on analysis of losses and adding
redundancy. We refer
\cite{Ait-Hellal et al 1999}, \cite{Altman and Jean-Marie 1998},
\cite{Cidon Khamisy and
Sidi 1993}, \cite{Dube Ait-Hellal and Altman 2003}, \cite{Dube and Altman
2003},
\cite{Gurewitz Sidi and Cidon 2000} to only mention a few. All of these
papers use
analytic techniques for the solution of one or other problem.

\section{New asymptotic results for the number of lost messages in
telecommunication systems}\label{Lost messages - new}

In this section we continue to study the asymptotic behavior of
the number of lost messages in optical telecommunication systems.
More specifically, in this section we apply the results of Sect.
\ref{Consecutive losses} to establish the asymptotic results for
consecutive refused messages of the model described in the
previous section. As in the previous section the phrase
\textit{refused message} is used to indicate that the message is
lost by overflowing the buffer, and in this section we provide new
results just for \textit{consecutive refused messages}. For the
number of $k$-consecutive refused messages we use the notation
$R_{\zeta(N),k}$, i.e. we only add the index $k$ to the previous
notation $R_{\zeta(N)}$ used in the previous section.

According to Lemma \ref{lem1}
$$
\mathrm{E}R_{\zeta(N)}=(\rho-1)\mathrm{E}P_{\zeta(N)},
$$
and
$$
\mathrm{E}P_{\zeta(N)}=\mu\mathrm{E}T_{\zeta(N)}
$$
(see rel. \eqref{5.10} and \eqref{5.12}), where $P_{\zeta(N)}$ is
the number of processed messages during a busy period
$T_{\zeta(N)}$.

The expectation $\mathrm{E}T_{\zeta(N)}$ is in turn found by using
the total probability formula
$$
\mathrm{E}T_{\zeta(N)}=\sum_{i=\zeta^{lower}}^{\zeta^{upper}}\mathrm{E}T_i\mathrm{P}\{\zeta(N)=i\}
$$
(see rel. \eqref{5.9}), where $T_i$ is the length of a busy period
in the $M/GI/1/i$ queueing system. The asymptotic analysis of
$\mathrm{E}T_{\zeta(N)}$ and $\mathrm{E}R_{\zeta(N)}$ is based on
asymptotic behavior of $\mathrm{E}T_n$ and the results of that
analysis for $\mathrm{E}R_{\zeta(N)}$ is given by Theorem
\ref{thm13}. Similarly to \eqref{5.9}, the above total probability
formula can be written for $\mathrm{E}R_{\zeta(N)}$, so we have:
$$
\mathrm{E}R_{\zeta(N)}=\sum_{i=\zeta^{lower}}^{\zeta^{upper}}\mathrm{E}L_i\mathrm{P}\{\zeta(N)=i\},
$$
where $L_i$ denotes the number of losses during a busy period of
the $M/GI/1/i$ queueing system. Therefore, the results of Sect.
\ref{Consecutive losses} can be applied immediately here.

\begin{thm}\label{thm-clapp} If $\rho<1$, then
$$
\lim_{N\to\infty}\mathrm{E}R_{\zeta(N),k}=0.
$$
If $\rho=1$ and $\rho_2=\lambda^2\int_0^\infty
x^2\mathrm{d}B(x)<\infty$, then
$$
\lim_{N\to\infty}\mathrm{E}R_{\zeta(N),k}=\frac{\sum_{i=1}^\infty
i\int_0^\infty \mathrm{e}^{-\lambda x}\frac{(\lambda
x)^{i+k}}{(i+k)!}\mathrm{d}B(x)}{\widehat{B}(\lambda)}.
$$
If $\rho>1$, then
$$
\lim_{N\to\infty}\frac{\mathrm{E}R_{\zeta(N),k}}{\mathrm{E}R_{\zeta(N)}}=\lim_{N\to\infty}\frac{\sum_{j=1}^N\varphi^{j-1}\sum_{i=N-j+k+1}^\infty\int_0^\infty\mathrm{e}^{-\lambda
x}\frac{(\lambda
x)^i}{i!}\mathrm{d}B(x)}{\sum_{j=1}^N\varphi^{j-1}\sum_{i=N-j+2}^\infty\int_0^\infty\mathrm{e}^{-\lambda
x}\frac{(\lambda x)^i}{i!}\mathrm{d}B(x)},
$$
where $\varphi$ is the least positive root of equation
\eqref{2.10}, and $\mathrm{E}R_{\zeta(N)}$ is defined by
asymptotic expansion \eqref{5.19} of Theorem \ref{thm13}.
\end{thm}

Other new results on asymptotic behavior of
$\mathrm{E}R_{\zeta(N),k}$ as $N\to\infty$ can be obtained by the
further analysis of the results obtained in Sect. \ref{Consecutive
losses} and \ref{Lost messages}. Similarly to Theorems \ref{thm15}
and \ref{thm16} we have the following theorem.

\begin{thm}
\label{thm15app} Let $\rho=1+\delta$, $\delta>0$, and $\delta\zeta(N)\to
 C\geq0$ a.s. as $N\to\infty$ and $\delta\to0$. Assume also that
 $\rho_3=\rho_3(N)$ is a bounded sequence, and there exists $\widetilde
 \rho_2=\lim_{N\to\infty}\rho_2(N)$. Then,
 \begin{equation*}
\mathrm{E}R_{\zeta(N),k}=\mathrm{e}^{2C/\widetilde{\rho}_2}\frac{\sum_{i=1}^\infty
i\int_0^\infty \mathrm{e}^{-\lambda x}\frac{(\lambda
x)^{i+k}}{(i+k)!}\mathrm{d}B(x)}{\widehat{B}(\lambda)}+o(1).
 \end{equation*}
 \end{thm}

\section{Optimal policies of using water in large dams: A simplified
model}\label{Dam1}

In this section we discuss an application of Tak\'acs' theorem to optimal
control of a
large dam. We consider a simple variant of this problem given in
\cite{Abramov 2007-2}.
The extended version corresponding to \cite{Abramov 2007 - water} will be
discussed in
the next section.

The upper and lower levels of a dam are denoted $L^{upper}$ and
$L^{lower}$
correspondingly, and the difference between the upper and lower level
$n=L^{upper}-L^{lower}$ characterizes the capacity of this dam. The value
$n$ is assumed
to be large, and this assumption enables us to use asymptotic analysis as
$n\to\infty$.

We assume that the units of water arriving at the dam are registered at
random instant
$t_1$, $t_2$,\ldots, and the interarrival times, $\tau_n=t_{n+1}-t_n$,
are mutually
independent, exponentially distributed random variables with parameter
$\lambda$.
Outflow of water is state dependent as follows. If the level of water is
between
$L^{lower}$ and $L^{upper}$ then the dam is said to be in normal state
and the duration
time between departure of a unit of water has the probability
distribution function
$B_1(x)$. If the level of water increases the level $L^{upper}$, then the
probability
distribution function of the interval between unit departures is
$B_2(x)$. If the level
of water is exactly $L^{lower}$, then the departure process of water is
frozen, and it
is resumed again as soon as the level of water exceeds the value
$L^{lower}$.

In terms of queueing theory, the problem can be formulated as follows.
Consider a
single-server queueing system with Poisson arrival stream of rate
$\lambda$. The service
time of customers depends upon queue-length as follows. If at the moment
of a beginning
of a service, the number of customers in the system is not greater than
$n$, then this
customer is served by probability distribution function $B_1(x)$.
Otherwise (if at the
moment of a beginning of a service, the number of customers in the system
exceeds $n$),
then the probability distribution function of the service time of this
customer is
$B_2(x)$. Notice, that the lower level $L^{lower}$ is associated with an
empty dam. The
dam specification of the problem is characterized by performance
criteria, which in
queueing formulation is as follows.

Let $Q(t)$ denote the (stationary) queue-length in time $t$. The problem
is to choose
the output parameter of the system (parameter of service time
distribution $B_1(x)$) so
to minimize the functional
\begin{equation}
\label{6.1} J(n)=p_1(n)J_1(n)+p_2(n)J_2(n),
\end{equation}
where $p_1(n)=\lim_{t\to\infty}\mathrm{P}\{Q(t)=0\}$,
$p_2(n)=\lim_{t\to\infty}\mathrm{P}\{Q(t)>n\}$, and $J_1(n)$, $J_2(n)$
are the
corresponding damage costs proportional to $n$. To be correct, assume
that $J_1(n)=j_1n$
and $J_2(n)=j_2n$, where $j_1$ and $j_2$ are given positive values.
Assuming that
$n\to\infty$ we will often write the $p_1$ and $p_2$ without index $n$.
The index $n$
will be also omitted in other functions such as $J_1$, $J_2$ and so on.

We will assume that the input rate $\lambda$ and probability distribution
function
$B_2(x)$ are given. The unknown parameter in the control problem is
associated with
probability distribution function $B_1(x)$. We will assume that
$B_1(x)=B_1(x,C)$ is the
family of probability distributions depending on parameter $C\geq0$, and
this parameter
$C$ is closely related to the expectation $\int_0^\infty
x\mathrm{d}B_1(x)$. Then the
output rate of the system in the normal state is associated with the
choice of parameter
$C\geq0$ resulting the choice of $B_1(x,C)$.

We use the notation $\frac{1}{\mu_i}=\int_0^\infty x\mathrm{d}B_i(x)$,
and
$\rho_i=\frac{\lambda}{\mu_i}$ for $i=1,2$. We assume that $\rho_2<1$,
and this last
assumption provides the stationary behavior of queueing system and
existence of required
limiting stationary probabilities $p_1$ and $p_2$ (independent of an
initial state of
the process). In addition, we assume existence of the third moment,
$\rho_{1,k}=\int_0^\infty(\lambda x)^k\mathrm{d}B_1(x)<\infty$, $k=2,3$.

For the above state dependent queueing system introduce the notation as
follows. Let
$T_n$, $I_n$, and $\nu_n$ respectively denote the duration of a busy
period, the
duration of an idle period, and the number of customers served during a
busy period. Let
$T_n^{(1)}$ and $T_n^{(2)}$ respectively denote the total times during a
busy period
when $0<Q(t)\leq n$ and $Q(t)>n$, and let $\nu_n^{(1)}$ and $\nu_n^{(2)}$
respectively
denote the total numbers of customers served during a busy period when
$0<Q(t)\leq n$
and $Q(t)>n$.

We have the following two equations
\begin{eqnarray}
\mathrm{E}T_n&=&\mathrm{E}T_n^{(1)}+\mathrm{E}T_n^{(2)},\label{6.2}\\
\mathrm{E}\nu_n&=&\mathrm{E}\nu_n^{(1)}+\mathrm{E}\nu_n^{(2)}.\label{6.3}
\end{eqnarray}

From Wald's equation \cite{Feller 1966}, p.384, we obtain:
\begin{eqnarray}
\mathrm{E}T_n^{(1)}&=&\frac{1}{\mu_1}\mathrm{E}\nu_n^{(1)},\label{6.4}\\
\mathrm{E}T_n^{(2)}&=&\frac{1}{\mu_2}\mathrm{E}\nu_n^{(2)}.\label{6.5}
\end{eqnarray}

The number of arrivals during a busy cycle coincides with the total
number of served
customers during a busy period. By using Wald's equation again we have
\begin{equation}
\label{6.6}
\begin{aligned}
\mathrm{E}\nu_n^{(1)}+\mathrm{E}\nu_n^{(2)}
&=\lambda\mathrm{E}T_{n}+\lambda\mathrm{E}I_{n}\\
&=\lambda\mathrm{E}T_{n}+1\\
&=\lambda\left(\mathrm{E}T_{n}^{(1)}+\mathrm{E}T_{n}^{(2)}\right)+1.
\end{aligned}
\end{equation}
Substituting \eqref{6.4} and \eqref{6.5} for the right-hand side of
\eqref{6.6}, we
obtain:
\begin{equation}
\label{6.7} \mathrm{E}\nu_n^{(1)}+\mathrm{E}\nu_n^{(2)}
=\rho_1\mathrm{E}\nu_n^{(1)}+\rho_2\mathrm{E}\nu_n^{(2)}.
\end{equation}
From \eqref{6.7} we arrive at
\begin{equation}
\label{6.8}
\mathrm{E}\nu_n^{(2)}=\frac{1}{1-\rho_2}-\frac{1-\rho_1}{1-\rho_2}\mathrm{E}\nu_n^{(1)},
\end{equation}
which expresses $\mathrm{E}\nu_n^{(2)}$ in terms of
$\mathrm{E}\nu_n^{(1)}$. For
example, if $\rho_1=1$ then for any $n\geq0$ we obtain
\begin{equation}
\label{6.9} \mathrm{E}\nu_n^{(2)}=\frac{1}{1-\rho_2}.
\end{equation}
That is in the case $\rho_1=1$, the value $\mathrm{E}\nu_n^{(2)}$ is the
same for all
$n\geq0$. This property is established in \cite{Abramov 1991} together
the
aforementioned property of losses in $M/GI/1/n$ queues (see relation
\eqref{2.8}).

Similarly, from \eqref{6.5} and \eqref{6.8} we obtain:
\begin{equation}
\label{6.10}
\mathrm{E}T_n^{(2)}=\frac{1}{\mu(1-\rho_2)}-\frac{1-\rho_1}{\mu(1-\rho_2)}\mathrm{E}T_n^{(1)}.
\end{equation}

From equations \eqref{6.8} and \eqref{6.10} one can obtain the stationary
probabilities
$p_1$ and $p_2$. By applying first renewal reward theorem (see e.g.
\cite{Ross 2000})
and then \eqref{6.6} and \eqref{6.8}, for $p_1$ and $p_2$ we obtain:
\begin{equation}
\label{6.11}
p_1=\frac{\mathrm{E}I_n}{\mathrm{E}T_n^{(1)}+\mathrm{E}T_n^{(2)}+\mathrm{E}I_n}
=\frac{1}{\mathrm{E}\nu_n^{(1)}+\mathrm{E}\nu_n^{(2)}}
=\frac{1-\rho_2}{1-(\rho_1-\rho_2)\mathrm{E}\nu_n^{(1)}},
\end{equation}
\begin{equation}
\label{6.12}
p_2=\frac{\mathrm{E}T_n^{(2)}}{\mathrm{E}T_n^{(1)}+\mathrm{E}T_n^{(2)}+\mathrm{E}I_n}
=\frac{\rho_2\mathrm{E}\nu_n^{(2)}}{\mathrm{E}\nu_n^{(1)}+\mathrm{E}\nu_n^{(2)}}
=\frac{\rho_2+\rho_2(\rho_1-1)\mathrm{E}\nu_n^{(1)}}{1+(\rho_1-\rho_2)\mathrm{E}\nu_n^{(1)}}.
\end{equation}

Thus, both probabilities $p_1$ and $p_2$ are expressed via
$\mathrm{E}\nu_n^{(1)}$. Our
next arguments are based on sample path arguments and the property of the
lack of memory
of exponential distribution. Using all of this, one concludes that the
random variable
$\nu_n^{(1)}$ coincides in distribution with the number of customers
served during a
busy period of the $M/GI/1/n$ queueing system. This characteristics has
been discussed
in Section 2. According to Wald's equation we have the representation
similar to
\eqref{2.1}. Specifically,
\begin{equation}\label{6.13}
\mathrm{E}\nu_n^{(1)}=\sum_{j=0}^{n}\mathrm{E}(\nu_{n-j+1}^{(1)})\int_0^\infty\mathrm{e}^{-\lambda
x}\frac{(\lambda x)^j}{j!}\mathrm{d}B_1(x),
\end{equation}
where $\nu_k^{(1)}$, $k=0,1,\ldots,n-1$, are the numbers of served
customers during a
busy period associated with similar queueing systems (i.e. having the
same rate
$\lambda$ of Poisson input and the same probability distribution
functions $B_1(x)$ and
$B_2(x)$ of the corresponding service times) but only defined by
parameters $k$
distinguished of $n$. As $n\to\infty$, the asymptotic behavior of
$\mathrm{E}\nu_n^{(1)}$ can be obtained from Tak\'acs' theorem. In turn,
the stationary
probabilities $p_1$ and $p_2$ are expressed via $\mathrm{E}\nu_n^{(1)}$,
and therefore
we have the following asymptotic theorem.

\begin{thm}
\label{thm22} If $\rho_1<1$, then
\begin{eqnarray}
\lim_{n\to\infty}p_1(n)&=&1-\rho_1,\label{6.14}\\
\lim_{n\to\infty}p_2(n)&=&0.\label{6.15}
\end{eqnarray}
If $\rho_1=1$, then
\begin{eqnarray}
\lim_{n\to\infty}np_1(n)&=&\frac{\rho_{1,2}}{2},\label{6.16}\\
\lim_{n\to\infty}np_2(n)&=&\frac{\rho_2}{1-\rho_2}\cdot\frac{\rho_{1,2}}{2}.\label{6.17}
\end{eqnarray}
If $\rho_1>1$, then
\begin{equation}
\label{6.18} \lim_{n\to\infty}\frac{p_1(n)}{\varphi^n}=
\frac{(1-\rho_2)[1+\lambda\widehat{B}_1^\prime(\lambda-\lambda\varphi)]}{\rho_1-\rho_2},
\end{equation}
where $\widehat{B}_1(s)$ is the Laplace-Stieltjes transform of $B_1(x)$,
and $\varphi$
is the least positive root of equation $z=\widehat{B}_1(\lambda-\lambda
z)$, and
\begin{equation}
\label{6.19}
\lim_{n\to\infty}p_2(n)=\frac{\rho_2(\rho_1-1)}{\rho_1-\rho_2}.
\end{equation}
\end{thm}

Let us discuss a relation of Theorem \ref{thm22} to optimal control
problem. Under the
assumption that $\rho_1<1$, relations \eqref{6.14} and \eqref{6.15}
enables us to
conclude as follows. The probability $p_1$ in positive in limit, while
the probability
$p_2$ vanishes. Therefore, in the case $\rho_1<1$ we have
$J\approx(1-\rho_1)J_1$, so
$J$ increases to infinity being asymptotically equivalent to
$(1-\rho_1)j_1n$. The
similar conclusion follows from \eqref{6.16} and \eqref{6.17} under the
assumption
$\rho_1>1$. In this case the probability $p_1$ vanishes, while the
probability $p_2$
tends to the positive limit $\frac{\rho_2(\rho_1-1)}{(\rho_1-\rho_2)}$.
This enables us
to conclude that $J$ increases to infinity together with $n$ increasing
to infinity
being asymptotically equivalent to
$\frac{\rho_2(\rho_1-1)}{(\rho_1-\rho_2)}j_2n$. In
the case $\rho_1=1$ both $p_1$ and $p_2$ vanish at rate
$O\left(\frac{1}{n}\right)$, and
$J$ therefore converges to the limit as $n\to\infty$. Thus, $\rho_1=1$ is
a possible
solution to the control problem, while the cases $\rho_1<1$ and
$\rho_1>1$ are
irrelevant. In the case $\rho_1=1$ we have as follows:
\begin{equation}
\label{6.20} \lim_{n\to\infty}J(n)=j_1\frac{\rho_{1,2}}{2}
+j_2\frac{\rho_2}{1-\rho_2}\cdot\frac{\rho_{1,2}}{2}.
\end{equation}

According to the result of \eqref{6.20}, the class of possible solutions
to the control
problem can be described by the following two cases, in both of which
$\delta\to0$
($\delta>0$) and $n\to\infty$: (i) $\rho_1=1+\delta$; (ii)
$\rho_1=1-\delta$.

In case (i) we have the following two theorems.

\begin{thm}
\label{thm23} Assume that $\rho_1=1+\delta$, $\delta>0$, and that
$n\delta\to C>0$ as
$\delta\to0$ and $n\to\infty$. Assume that $\rho_{1,3}=\rho_{1,3}(n)$ is
a bounded
sequence and that the limit
$\widetilde{\rho}_{1,2}=\lim_{n\to\infty}\rho_{1,2}(n)$
exists. Then,
\begin{eqnarray}
p_1&=&\frac{\delta}{\mathrm{e}^{2C/\widetilde{\rho}_{1,2}}-1}+o(\delta),\label{6.21}\\
p_2&=&\frac{\delta\rho_2\mathrm{e}^{2C/\widetilde{\rho}_{1,2}}}
{(1-\rho_2)(\mathrm{e}^{2C/\widetilde{\rho}_{1,2}}-1)}+o(\delta).\label{6.22}
\end{eqnarray}
\end{thm}

\begin{thm}
\label{thm24} Under the conditions of Theorem \ref{thm23}, assume that
$C=0$. Then,
\begin{eqnarray}
\lim_{n\to\infty}np_1(n)&=&\frac{\rho_{1,2}}{2},\label{6.23}\\
\lim_{n\to\infty}np_2(n)&=&\frac{\rho_2}{1-\rho_2}\cdot\frac{\rho_{1,2}}{2}.\label{6.24}
\end{eqnarray}
\end{thm}

The proof of these two theorems is based on the above expansions given by
\eqref{2.12}
and \eqref{2.13}, and thus the proof is similar to the proof of
aforementioned Theorems
\ref{thm2} and \ref{thm3}. For details of the proofs see \cite{Abramov
2007-2}.

In case (ii) we have the following two theorems.

\begin{thm}
\label{thm25} Assume that $\rho_1=1-\delta$, $\delta>0$, and that
$n\delta\to C>0$ as
$\delta\to0$ and $n\to\infty$. Assume that $\rho_{1,3}=\rho_{1,3}(n)$ is
a bounded
sequence, and that the limit
$\widetilde{\rho}_{1,2}=\lim_{n\to\infty}\rho_{1,2}(n)$
exists. Then,
\begin{eqnarray}
p_1&=&\delta\mathrm{e}^{\widetilde{\rho}_{1,2}/2C}+o(\delta),\label{6.25}\\
p_2&=&\delta\frac{\rho_2}{1-\rho_2}(\mathrm{e}^{\widetilde{\rho}_{1,2}}-1)+o(\delta)
\label{6.26}.
\end{eqnarray}
\end{thm}

\begin{thm}
\label{thm26} Under the conditions of Theorem \ref{thm25}, assume that
$C=0$. Then we
obtain \eqref{6.23} and \eqref{6.24}.
\end{thm}

Theorems \ref{thm25} and \ref{thm26} do not follow directly from other
theorems as this
was in the case of Theorem \ref{thm23}, which is derived from \ref{thm22}
by using
appropriate asymptotic expansions. The proof of Theorem \ref{thm25} is
based on
Tauberian theorem of Hardy and Littlewood which can be found in many
sources (e.g.
\cite{Postnikov 1980}, \cite{Subhankulov 1976}, \cite{Sznajder and Filar
1992},
\cite{Takacs 1967}, p. 203 and \cite{Widder 1941}). Below we give the
proof of Theorem
\ref{thm25} based on the aforementioned Tauberian theorem.

We have the representation
\begin{equation*}\label{6.27}
\sum_{n=0}^\infty\mathrm{E}\nu_n^{(1)}z^n=\frac{\widehat{B}_1(\lambda-\lambda
z)}{\widehat{B}_1(\lambda-\lambda z)-z},
\end{equation*}
which is the consequence of \eqref{1.2}. The sequence
$\left\{\mathrm{E}\nu_n^{(1)}\right\}$ is increasing, and, for
$\rho_1=1$, from the
aforementioned Tauberian theorem of Hardy and Littlewood we have
\begin{equation*}\label{6.28}
\lim_{n\to\infty}\frac{\mathrm{E}\nu_n^{(1)}}{n}=\lim_{z\uparrow1}(1-z)^2
\frac{\widehat{B}_1(\lambda-\lambda z)}{\widehat{B}_1(\lambda-\lambda
z)-z}.
\end{equation*}

In the case $\rho_1=1-\delta$ and $\delta n\to C$ as $n\to\infty$ and
$\delta\to0$,
according to the same Tauberian theorem of Hardy and Littlewood, the
asymptotic behavior
of $\mathrm{E}\nu_n^{(1)}$ can be found from the asymptotic expansion of
\begin{equation}
\label{6.29} (1-z)\frac{\widehat{B}(\lambda-\lambda
z)}{\widehat{B}(\lambda-\lambda
z)-z},
\end{equation}
as $z\uparrow1$. Expanding the denominator of \eqref{6.29} to the Taylor
series, we
obtain:
\begin{equation}
\label{6.30}
\begin{aligned}
\frac{1-z}{\widehat{B}(\lambda-\lambda
z)-z}&=\frac{1-z}{1-z-\rho_1(1-z)+(\widetilde{\rho}_{1,2}/2)(1-z)^2+O((1-z)^3)}\\
&=\frac{1}{\delta+(\widetilde{\rho}_{1,2}/2)(1-z)+O((1-z)^2)}\\
&=\frac{1}{\delta[1+(\widetilde{\rho}_{1,2}/2\delta)(1-z)+O((1-z)^2)]}\\
&=\frac{1}{\delta\exp((\widetilde{\rho}_{1,2}/2\delta)(1-z))}(1+o(1)).
\end{aligned}
\end{equation}
Therefore, assuming that $z=(n-1)/n\to 1$ as $n\to\infty$, from
\eqref{6.30} we have:
\begin{equation}
\label{6.31}
\mathrm{E}\nu_n^{(1)}=\frac{1}{\delta\mathrm{e}^{\widetilde{\rho}_{1,2}/2C}}(1+o(1)).
\end{equation}
Substituting \eqref{6.31} into \eqref{6.11} and \eqref{6.12} we arrive at
the desired
statement of Theorem \ref{thm25}. The proof of Theorem \ref{thm26} is
similarly based on
the above expansion.

Above Theorems \ref{thm23} - \ref{thm26} enable us to solve the control
problem. We have
the following limiting relation:
\begin{equation}\label{6.32}
\begin{aligned}
\lim_{n\to\infty}J(n)&=\lim_{n\to\infty}[p_1(n)J_1(n)+p_2(n)J_2(n)]\\
&=j_1\lim_{n\to\infty}np_1(n)+j_2\lim_{n\to\infty}np_2(n).
\end{aligned}
\end{equation}

According to the cases (i) and (ii) we have two corresponding
functionals. Specifically,
substituting \eqref{6.21} and \eqref{6.22} into the right-hand side of
\eqref{6.32} and
taking into account that $n\delta\to C$ we obtain:
\begin{equation}\label{6.33}
J^{upper}=C\left[j_1\frac{1}{\mathrm{e}^{2C/\widetilde{\rho}_{1,2}}-1}
+j_2\frac{\rho_2\mathrm{e}^{2C/\widetilde{\rho}_{1,2}}}{(1-\rho_2)
(\mathrm{e}^{2C/\widetilde{\rho}_{1,2}}-1)}\right].
\end{equation}

Next, substituting \eqref{6.25} and \eqref{6.26} into the right-hand side
of
\eqref{6.32} and taking into account that $n\delta\to C$ we obtain:
\begin{equation}\label{6.34}
J^{lower}=C\left[j_1\mathrm{e}^{\widetilde{\rho}_{1,2}/2C}
+j_2\frac{\rho_2}{1-\rho_2}(\mathrm{e}^{\widetilde{\rho}_{1,2}/2C}-1)\right].
\end{equation}

An elementary analysis of functionals \eqref{6.33} and \eqref{6.34}
shows, that the
minimum of the both of them is achieved under $C=0$ if and only if
\begin{equation}
\label{6.35} j_1=j_2\frac{\rho_2}{1-\rho_2}.
\end{equation}

More detailed analysis of these functionals \eqref{6.33} and \eqref{6.34}
(see
\cite{Abramov 2007-2}) leads to the following solution to the control
problem.

\begin{thm}
\label{thm27} If the parameters $\lambda$ and $\rho_2$ are given, then
the optimal
solution to the control problem is as follows.

If
\begin{equation*}
j_1=j_2\frac{\rho_2}{1-\rho_2},
\end{equation*}
then the optimal solution to the control problem is achieved for
$\rho_1=1$.

If
\begin{equation*}
j_1>j_2\frac{\rho_2}{1-\rho_2},
\end{equation*}
then the optimal solution to the control problem is a minimization of the
functional
$J^{upper}$. The optimal solution is achieved for $\rho_1=1+\delta$,
where $\delta(n)$
is a small positive parameter and $n\delta(n)\to C$, the nonnegative
parameter
minimizing \eqref{6.33}.

If
\begin{equation*}
j_1<j_2\frac{\rho_2}{1-\rho_2},
\end{equation*}
then the optimal solution to the control problem is a minimization of the
functional
$J^{lower}$. The optimal solution is achieved for $\rho_1=1-\delta$,
where $\delta(n)$
is a small positive parameter and $n\delta(n)\to C$, the nonnegative
parameter
minimizing \eqref{6.34}.
\end{thm}

There is a large number of papers in the dam literature concerning
different optimal
control problem of water of storage resources. To indicate only a few of
them we refer
\cite{Abdel-Hameed 2000}, \cite{Abdel-Hameed and Nakhi 1990}, \cite{Bae
Kim and Lee
2003}, \cite{Faddy 1974}, \cite{Lam and Lou 1987}, \cite{Lee and Ahn
1998} and
\cite{Zukerman 1977}. However, the optimal control problem of dams and
its solution,
which is discussed in this section of the paper, specifically differs
from all of the
earlier considerations known from the literature. In the next section we
discuss a more
extended control problem of large dams.

\section{Optimal policies of using water in large dams: An extended
model}\label{Dam2}

In the previous section we found optimal solution to the control problem
of minimization
the objective function
\begin{equation*}
J=p_1J_1+p_2J_2,
\end{equation*}
where $p_1$ and $p_2$ are stationary probabilities of passage across the
lower and upper
levels of a dam, i.e.
\begin{eqnarray*}
p_1&=&\lim_{t\to\infty}\mathrm{P}\{L_t=L^{lower}\},\\
p_2&=&\lim_{t\to\infty}\mathrm{P}\{L_t>L^{upper}\},
\end{eqnarray*}
$L_t$ denotes the water level in time $t$, and $J_1$, $J_2$ are the cost
functions
having the form: $J_1=j_1n$ and $J_2=j_2n$. The extended problem
considered in this
section is the problem of minimization of the functional
\begin{equation}
J=p_1J_1+p_2J_2+\sum_{i=L^{lower}+1}^{L^{upper}}c_iq_i,
\end{equation}
where
\begin{equation}
\label{7.5} q_i=\lim_{t\to\infty}\mathrm{P}\{L_t=L^{lower}+i\}, \
i=1,2,\ldots,n,
\end{equation}
and $c_i$ is the water cost when the level of dam is equal to $i$. The
costs $c_i$ are
assumed to be decreasing in $i$, i.e. $c_{i+1}\leq c_i$ for all $1\leq
i\leq n-1$. (The
water is cheaper when the dam level is higher.)

In queueing formulation the level $L^{lower}$ is equated with an empty
queue. The
queueing formulation of the dam model is given in the previous section.
All of the
assumptions for that state-dependent queueing system remains the same.

It was shown in the previous section that the probabilities $p_1$ and
$p_2$ are
expressed via $\mathrm{E}\nu_n^{(1)}$ (relations \eqref{6.11} and
\eqref{6.12}), and the
asymptotic representations of $p_1$ and $p_2$ is derived from the
asymptotic formula of
$\mathrm{E}\nu_n^{(1)}$. The stationary probabilities $q_i$,
$i=1,2,\ldots,n$ can be
obtained from renewal arguments (see e.g. Ross \cite{Ross 2000}). Namely,
for
$i=1,2,\ldots,n$ we have:
\begin{equation}
\label{7.6}
q_i=\frac{\mathrm{E}T_i^{(1)}-\mathrm{E}T_{i-1}^{(1)}}{\mathrm{E}T_n+\mathrm{E}I_n}.
\end{equation}
($\mathrm{E}T_i^{(1)}$ means the expectation of the total time
 that the customers are served by probability
distribution function $B_1(x)$ during a busy period of the
state-dependent system, which
is distinguished from the described system only by parameter $i$ given
instead of the
original parameter $n$. According to sample-path arguments and the
property of the lack
of memory of exponential distribution (these arguments are given in the
previous
section), the random variable $T_i^{(1)}$ coincides in distribution with
a busy period
of the $M/GI/1/i$ queueing system.) The probabilities $q_i$,
$i=1,2,\ldots,n$, given by
\eqref{7.6}, can be also rewritten
\begin{equation}
\label{7.7}
q_i=\rho_1\frac{\mathrm{E}\nu_i^{(1)}-\mathrm{E}\nu_{i-1}^{(1)}}{\mathrm{E}\nu_n}.
\end{equation}
The equivalence of \eqref{7.6} and \eqref{7.7} follows easily from the
equations
$\mathrm{E}\nu_i^{(1)}=\mu_1\mathrm{E}T_i^{(1)}$ for all
$i=0,1,\ldots,n$, and
$\mathrm{E}\nu_i^{(1)}=\lambda\mathrm{E}T_n+\lambda\mathrm{E}I_n$, which
are Wald's
equations.

Representation \eqref{7.7} can also be rewritten in other forms, which
are more
convenient for our purposes. Recall that (see relation \eqref{6.8})
\begin{equation*}
\label{7.8}
\mathrm{E}\nu_n^{(2)}=\frac{1}{1-\rho_2}-\frac{1-\rho_1}{1-\rho_2}\mathrm{E}\nu_n^{(1)},
\end{equation*}
and therefore, taking into account that
$\mathrm{E}\nu_n=\mathrm{E}\nu_n^{(1)}+\mathrm{E}\nu_n^{(2)}$, we also
have:
\begin{equation}
\label{7.9}
\mathrm{E}\nu_n=\frac{1}{1-\rho_2}+\frac{\rho_1-\rho_2}{1-\rho_2}\mathrm{E}\nu_n^{(1)}.
\end{equation}
Then elementary substitution of \eqref{7.9} into \eqref{7.7} gives us
\begin{equation}
\label{7.10}
q_i=\frac{\rho_1(1-\rho_2)}{1+(\rho_1-\rho_2)\mathrm{E}\nu_n^{(1)}}
\left(\mathrm{E}\nu_i^{(1)}-\mathrm{E}\nu_{i-1}^{(1)}\right), \
i=1,2,\ldots,n.
\end{equation}

Comparison with \eqref{6.11} enables us to rewrite \eqref{7.10} in the
other form:
\begin{equation}
\label{7.11}
q_i=\rho_1p_1\left(\mathrm{E}\nu_i^{(1)}-\mathrm{E}\nu_{i-1}^{(1)}\right),
\
i=1,2,\ldots,n.
\end{equation}

Let us now discuss the different cases as (i) $\rho_1=1$, and for
sufficiently small
$\delta>0$ (ii) $\rho_1=1+\delta$, and (iii) $\rho_1=1-\delta$.

We start from case (i). In this case we have the following result.
\begin{thm}\label{thm28}
If $\rho_1=1$ and $\rho_2<\infty$, then for all $i=1,2,\ldots,n$
\begin{equation}
\label{7.12}  \lim_{n\to\infty}nq_{n-i}=1.
\end{equation}
\end{thm}
The proof of this theorem is based on application Tauberian Theorem
\ref{thm2Postnikov}.
Namely, from that Tauberian theorem we have as follows.
\begin{lem}
\label{lem2} Under the conditions of Theorem \ref{thm28} for any $j\geq0$
we have
\begin{equation}
\label{7.13}
\mathrm{E}\nu_{n-j}^{(1)}-\mathrm{E}\nu_{n-j-1}^{(1)}=\frac{2}{\rho_{1,2}}+o(1),
\end{equation}
as $n\to\infty$.
\end{lem}
Then, the statement of Theorem \ref{thm28} follows by application
of Lemma \ref{lem2}, i.e. by substitution \eqref{7.13} into
\eqref{7.11} and consequent application of relation \eqref{6.16}
of Theorem \ref{thm22}. Notice, that just Tauberian theorem
\ref{thm2Postnikov} is applied here, since the appropriate
statement of Tak\'acs' theorem is not enough in order to prove the
required statement of Theorem \ref{thm28}.

In turn, Theorem \ref{thm28} leads to the following result.
\begin{prop}
\label{prop1} Under the conditions of Theorem \ref{thm28} we have
\begin{equation}
\label{7.14}  \lim_{n\to\infty}
J(n)=j_1\frac{\rho_{1,2}}{2}+j_2\frac{\rho_2}{1-\rho_2}\cdot\frac{\rho_{1,2}}{2}+c^*,
\end{equation}
where
\begin{equation*}
c^*=\lim_{n\to\infty}\frac{1}{n}\sum_{i=1}^n c_i.
\end{equation*}
\end{prop}

Notice, that the only difference between \eqref{6.20} and representation
\eqref{7.14} is
in the presence of the term $c^*$ in \eqref{7.14}.

In case (ii) we have the following theorem.

\begin{thm}
\label{thm29} Assume that $\rho_1=1+\delta$, $\delta>0$ and $n\delta\to
C>0$ as
$\delta\to0$ and $n\to\infty$. Assume also that $\rho_{1,3}(n)$ is a
bounded sequence
and there exists $\widetilde{\rho}_{1,2}=\lim_{n\to\infty}\rho_{1,2}(n)$.
Then, for any
$j\geq0$
\begin{equation}
\label{7.16} q_{n-j}=\frac{\mathrm{e}^{2C/\widetilde{\rho}_{1,2}}}
{\mathrm{e}^{2C/\widetilde{\rho}_{1,2}}-1}
\left(1-\frac{2\delta}{\widetilde{\rho}_{1,2}}\right)^j\frac{2\delta}
{\widetilde{\rho}_{1,2}}+o(\delta).
\end{equation}
\end{thm}

The proof of this theorem is based on the expansion
\begin{equation*}
\label{7.17} \mathrm{E}\nu_{n-j}^{(1)}=\frac{\varphi^j}
{\varphi^n\left[1+\lambda\widehat{B}_1^\prime(\lambda-\lambda\varphi)\right]}
+\frac{1}{1-\rho_1}+o(1),
\end{equation*}
which is the consequence of an application of Tak\'acs' theorem to
recurrence relation
\eqref{6.13}, as well as expansions \eqref{2.12} and \eqref{2.13}.

From Theorem \ref{thm29} we have the following result.
\begin{prop}
\label{prop2} Under the assumptions of Theorem \ref{thm29} let us denote
the objective
function $J$ by $J^{upper}$. We have the representation
\begin{equation}
\label{7.18}
J^{upper}=C\left[j_1\frac{1}{\mathrm{e}^{2C/\widetilde{\rho}_{1,2}}-1}
+j_2\frac{\rho_2\mathrm{e}^{2C/\widetilde{\rho}_{1,2}}}{(1-\rho_2)
(\mathrm{e}^{2C/\widetilde{\rho}_{1,2}}-1)}\right]+c^{upper},
\end{equation}
where
\begin{equation}\label{7.19}
c^{upper}=\frac{2C}{\widetilde{\rho}_{1,2}}\cdot\frac{\mathrm{e}^{2C/\widetilde{\rho}_{1,2}}}
{\mathrm{e}^{2C/\widetilde{\rho}_{1,2}}-1}
\lim_{n\to\infty}\widehat{C}_n\left(1-\frac{2C}{\widetilde{\rho}_{1,2}n}\right),
\end{equation}
and $\widehat{C}_n(z)=\sum_{j=0}^{n-1}c_{n-j}z^j$ is a backward
generating cost
function.
\end{prop}

For details of the proof see \cite{Abramov 2007 - water}.

Notice, that the term
\begin{equation*}\label{7.20}
C\left[j_1\frac{1}{\mathrm{e}^{2C/\widetilde{\rho}_{1,2}}-1}
+j_2\frac{\rho_2\mathrm{e}^{2C/\widetilde{\rho}_{1,2}}}{(1-\rho_2)
(\mathrm{e}^{2C/\widetilde{\rho}_{1,2}}-1)}\right]
\end{equation*}
is the value of the function $J^{upper}$ for the model where the water
costs have not
been taken into account (see relation \eqref{6.33}). So, the function
$c^{upper}$ given
by \eqref{7.19} is a new element for the function $J^{upper}$.

Case (iii) is more delicate. The additional assumption that required here
is that the
class of probability distributions $\{B_1(x)\}$ is such that there exists
a unique root
$\tau>1$ of the equation
\begin{equation}
\label{7.21} z=\widehat{B}(\lambda-\lambda z),
\end{equation}
and there also exists the derivative
$\widehat{B}^\prime(\lambda-\lambda\tau)$. In
general, under the assumption that $\rho_1<1$ the root of equation
\eqref{7.21} not
necessarily exists. Such type of condition has been considered by Willmot
\cite{Willmot
1988} in order to obtain the asymptotic behavior for the probability of
high-level
queue-length in stationary $M/GI/1$ queueing systems. Let $q_i[M/GI/1]$,
$i=0,1,\ldots$
denote the stationary queue-length probabilities. Willmot \cite{Willmot
1988} showed
that
\begin{equation}
\label{7.22} q_i[M/GI/1]=\frac{(1-\rho_1)(1-\tau)}{\tau^i
[1+\lambda\widehat{B}^\prime(\lambda-\lambda\tau)]}[1+o(1)],
\end{equation}
as $i\to\infty$. On the other hand, there is the following
representation:
\begin{equation}
\label{7.23}
q_i[M/GI/1]=(1-\rho_1)\left(\mathrm{E}\nu_i^{(1)}-\mathrm{E}\nu_{i-1}^{(1)}\right),
\
i=1,2,\ldots,
\end{equation}
which agrees with the well-known Pollaczek-Khintchin formula (e.g.
Tak\'acs \cite{Takacs
1962}, p. 242).

From \eqref{7.22} and \eqref{7.23} for any $j\geq0$ we have the following
asymptotic
proportion:
\begin{equation}
\label{7.24} \frac{\mathrm{E}\nu_{n-j}-\mathrm{E}\nu_{n-j-1}}
{\mathrm{E}\nu_n^{(1)}-\mathrm{E}\nu_{n-1}^{(1)}}=\tau^j[1+o(1)].
\end{equation}

In order to formulate and prove a theorem on asymptotic behavior of
stationary
probabilities $q_i$ for case (iii) we assume that the class of
probability distributions
$\{B_1(x)\}$ is as follows. Under the assumption that $\rho_1=1-\delta$,
$\delta>0$, and
$\delta\to0$ and $n\to\infty$, we assume that there exists the value
$\epsilon_0>0$
small enough (proportionally to $\delta$) such that for all
$0\leq\epsilon\leq\epsilon_0$ the family of probability distributions
$B_{1,\epsilon}(x)$, provided now by parameter $\epsilon$, satisfies the
following
condition: its Laplace-Stieltjes transform $\widehat{B}_{1,\epsilon}(z)$
is an analytic
function in a small neighborhood of zero and
\begin{equation}\label{7.24+}
\widehat B_{1,\epsilon}^\prime(\lambda\epsilon)<\infty.
\end{equation}

We have the following theorem.
\begin{thm}
\label{thm30} Assume that the class of probability distributions
$B_{1,\epsilon}(x)$ is
defined according to the above convention and satisfies \eqref{7.24+}.
Assume that
$\rho_1=1-\delta$, $\delta>0$, and $n\delta\to C>0$ as $\delta>0$ and
$n\to\infty$.
Assume that $\rho_{1,3}(n)$ is a bounded sequence and there exists
$\widetilde{\rho_2}=\lim_{n\to\infty}\rho_{1,2}(n)$. Then,
\begin{equation}
\label{7.25} q_{n-j}=\frac{2\delta}{\widetilde{\rho}_{1,2}}\cdot
\frac{1}{\mathrm{e}^{2C/\widetilde{\rho}_{1,2}}-1}\left(1+\frac{2\delta}
{\widetilde{\rho}_{1,2}}\right)^j[1+o(1)]
\end{equation}
for any $j\geq0$.
\end{thm}

The proof of this theorem uses the expansion
\begin{equation*}
\label{7.26} \tau=1+\frac{2\delta}{\widetilde{\rho}_{1,2}}+O(\delta^2),
\end{equation*}
which is similar to the expansion of \eqref{2.12} (for more details see
\cite{Abramov
2007 - water}).

From this theorem we arrive at the following proposition.

\begin{prop}
\label{prop3} Under the assumptions of Theorem \ref{thm30} denote the
objective function
$J$ by $J^{lower}$. We have the following representation:
\begin{equation}
\label{7.27}
J^{lower}=C\left[j_1\mathrm{e}^{\widetilde{\rho}_{1,2}/2C}+j_2\frac{\rho_2}{1-\rho_2}
\left(\mathrm{e}^{\widetilde{\rho}_{1,2}/2C}-1\right)\right]+c^{lower},
\end{equation}
where
\begin{equation}
\label{7.28} c^{lower}=\frac{2C}{\widetilde{\rho}_{1,2}}\cdot
\frac{1}{\mathrm{e}^{2C/\widetilde{\rho}_{1,2}}-1}
\lim_{n\to\infty}\frac{1}{n}\widehat{C}_n
\left(1+\frac{2C}{\widetilde{\rho}_{1,2}n}\right),
\end{equation}
and $\widehat{C}(z)=\sum_{j=0}^{n-1}c_{n-j}z^j$ is a backward generating
cost function.
\end{prop}

For details of the proof see \cite{Abramov 2007 - water}.

Notice, that the term
\begin{equation*}\label{7.29}
C\left[j_1\mathrm{e}^{\widetilde{\rho}_{1,2}/2C}+j_2\frac{\rho_2}{1-\rho_2}
\left(\mathrm{e}^{\widetilde{\rho}_{1,2}/2C}-1\right)\right]
\end{equation*}
is the value of the function $J^{lower}$ for the model where the water
costs have not
been taken into account (see relation \eqref{6.34}). So, the function
$c^{lower}$ given
by \eqref{7.28} is a new element for the function $J^{lower}$.

Representations \eqref{7.19} and \eqref{7.28} are not convenient, and for
the purpose of
the further analysis we provide other representations.

Introduce the following functions:
\begin{equation}
\label{7.30}
\psi(C)=\lim_{n\to\infty}\frac{\sum_{j=0}^{n-1}c_{n-j}\left(1-\frac{2C}
{\widetilde{\rho}_{1,2}n}\right)^j}{\sum_{j=0}^{n-1}\left(1-\frac{2C}
{\widetilde{\rho}_{1,2}n}\right)^j},
\end{equation}
\begin{equation}
\label{7.31}
\eta(C)=\lim_{n\to\infty}\frac{\sum_{j=0}^{n-1}c_{n-j}\left(1+\frac{2C}
{\widetilde{\rho}_{1,2}n}\right)^j}{\sum_{j=0}^{n-1}\left(1+\frac{2C}
{\widetilde{\rho}_{1,2}n}\right)^j}.
\end{equation}
Since ${c_i}$ is a bounded sequence, then the both limits of \eqref{7.30}
and
\eqref{7.31} do exist. Immediate algebraic calculations show that
$c^{upper}=\psi(C)$
and $c^{lower}=\eta(C)$ if one substitutes the exact representation for
the
corresponding limits:
\begin{equation*}
\label{7.32}
\lim_{n\to\infty}\frac{1}{n}\sum_{j=0}^{n-1}\left(1-\frac{2C}
{\widetilde{\rho}_{1,2}n}\right)^j=\frac{\widetilde{\rho}_{1,2}}{2C}
\left(1-\mathrm{e}^{-2C/\widetilde{\rho}_{1,2}}\right),
\end{equation*}
and
\begin{equation*}
\label{7.33}
\lim_{n\to\infty}\frac{1}{n}\sum_{j=0}^{n-1}\left(1+\frac{2C}
{\widetilde{\rho}_{1,2}n}\right)^j=\frac{\widetilde{\rho}_{1,2}}{2C}
\left(\mathrm{e}^{2C/\widetilde{\rho}_{1,2}}-1\right).
\end{equation*}

The functions $\psi(C)$ and $\eta(C)$ satisfy the following properties.
The function
$\psi(C)$ is a decreasing function, and its maximum is $\psi(0)=c^*$. The
function
$\eta(C)$ is an increasing function, and its minimum is $\eta(0)=c^*$.
Recall that
$c^*=\lim_{n\to\infty}\frac{1}{n}\sum_{i=1}^n c_i$.

The proof of these properties is based on the following properties of
numerical series
(see Hardy, Littlewood and Polya \cite{Hardy Littlewood and Polya 1952}
or Marschall and
Olkin \cite{Marschall and Olkin 1979}). Let $\{a_n\}$ and $\{b_n\}$ be
arbitrary
sequences of numbers. If one of them is increasing but another is
decreasing, then for
any finite sum
\begin{equation}\label{7.34}
\sum_{n=1}^la_nb_n\leq\frac{1}{l}\sum_{n=1}^la_n\sum_{n=1}^lb_n.
\end{equation}
If both of these sequences are increasing or decreasing, then
\begin{equation}\label{7.35}
\sum_{n=1}^la_nb_n\geq\frac{1}{l}\sum_{n=1}^la_n\sum_{n=1}^lb_n.
\end{equation}

Rewrite \eqref{7.30} and \eqref{7.31} as follows:
\begin{equation}
\label{7.36} \lim_{n\to\infty}\frac{1}{n}\sum_{j=1}^{n-1}c_{n-j}
\left(1-\frac{2C}{\widetilde{\rho}_{1,2}n}\right)^j=\psi(C)
\lim_{n\to\infty}\sum_{j=1}^{n-1}
\left(1-\frac{2C}{\widetilde{\rho}_{1,2}n}\right)^j,
\end{equation}
and
\begin{equation}
\label{7.37} \lim_{n\to\infty}\frac{1}{n}\sum_{j=1}^{n-1}c_{n-j}
\left(1+\frac{2C}{\widetilde{\rho}_{1,2}n}\right)^j=\eta(C)
\lim_{n\to\infty}\sum_{j=1}^{n-1}
\left(1+\frac{2C}{\widetilde{\rho}_{1,2}n}\right)^j.
\end{equation}

Applying inequality \eqref{7.34} to the left-hand side of \eqref{7.36}
and letting
$n\to\infty$, we obtain:
\begin{equation}
\label{7.38}
\begin{aligned}
 &\lim_{n\to\infty}\frac{1}{n}\sum_{j=0}^{n-1}c_{n-j}
\left(1-\frac{2C}{\widetilde{\rho}_{1,2}n}\right)^j\\
&\leq\lim_{n\to\infty}\frac{1}{n}\sum_{j=1}^{n-1}c_{n-j}\lim_{n\to\infty}
\frac{1}{n}\sum_{j=1}^{n-1}\left(1-\frac{2C}{\widetilde{\rho}_{1,2}n}\right)^j\\
&=\psi(0)\lim_{n\to\infty}
\frac{1}{n}\sum_{j=1}^{n-1}\left(1-\frac{2C}{\widetilde{\rho}_{1,2}n}\right)^j.
\end{aligned}
\end{equation}
Comparing \eqref{7.36} with \eqref{7.38} we arrive at
\begin{equation*}
\label{7.39} \psi(0)=c^*\geq\psi(C),
\end{equation*}
i.e. $\psi(0)=c^*$ is the maximum value of $\psi(C)$.

To show that $\psi(C)$ is a decreasing function, we are actually to show
that for any
nonnegative $C_1\leq C$ we have $\psi(C)\leq\psi(C_1)$. The proof of this
fact is based
on the following asymptotic relation. For small positive $\delta_1$ and
$\delta_2$ we
have $1-\delta_1-\delta_2=(1-\delta_1)(1-\delta_2)+O(\delta_1\delta_2)$.
In out terms,
this means that for $C_1<C$ we have:
\begin{equation*}\label{7.40}
1-\frac{2C}{\widetilde{\rho}_{1,2}n}=\left(1-\frac{2C_1}{\widetilde{\rho}_{1,2}n}\right)
\left(1-\frac{2C-2C_1}{\widetilde{\rho}_{1,2}n}\right)+O\left(\frac{1}{n^2}\right).
\end{equation*}

Therefore, similarly to \eqref{7.38} for the left-hand side of
\eqref{7.36} we
have:\footnote{This is a very short version of the original proof given
in \cite{Abramov
2007 - water}. If $c_i\equiv c$, i.e. all coefficients are equal to the
same constant,
then the inequality in \eqref{7.41} becomes the equality. For the
following technical
details associated with this property see \cite{Abramov 2007 - water}.}
\begin{equation}\label{7.41}
\begin{aligned}
&\lim_{n\to\infty}\frac{1}{n}\sum_{j=0}^{n-1}c_{n-j}
\left(1-\frac{2C}{\widetilde{\rho}_{1,2}n}\right)^j\\
&=\lim_{n\to\infty}\frac{1}{n}\sum_{j=0}^{n-1}c_{n-j}
\left(1-\frac{2C_1}{\widetilde{\rho}_{1,2}n}\right)^j
\left(1-\frac{2C-2C_1}{\widetilde{\rho}_{1,2}n}\right)^j\\
&\leq\lim_{n\to\infty}\frac{1}{n}\sum_{j=0}^{n-1}c_{n-j}
\left(1-\frac{2C_1}{\widetilde{\rho}_{1,2}n}\right)^j
\lim_{n\to\infty}\frac{1}{n}\sum_{j=0}^{n-1}
\left(1-\frac{2C-2C_1}{\widetilde{\rho}_{1,2}n}\right)^j\\
&=\psi(C_1)\lim_{n\to\infty}\frac{1}{n}\sum_{j=0}^{n-1}
\left(1-\frac{2C_1}{\widetilde{\rho}_{1,2}n}\right)^j
\lim_{n\to\infty}\frac{1}{n}\sum_{j=0}^{n-1}
\left(1-\frac{2C-2C_1}{\widetilde{\rho}_{1,2}n}\right)^j.
\end{aligned}
\end{equation}

On the other hand, by using inequality \eqref{7.35} for the right-hand
side of
\eqref{7.36} we have:
\begin{equation}
\label{7.42}
\begin{aligned}
&\psi(C)\lim_{n\to\infty}\frac{1}{n}\sum_{j=0}^{n-1}
\left(1-\frac{2C}{\widetilde{\rho}_{1,2}n}\right)^j\\
&=\psi(C)\lim_{n\to\infty}\frac{1}{n}\sum_{j=0}^{n-1}
\left(1-\frac{2C_1}{\widetilde{\rho}_{1,2}n}\right)^j
\left(1-\frac{2C-2C_1}{\widetilde{\rho}_{1,2}n}\right)^j\\
&\geq\psi(C)\lim_{n\to\infty}\frac{1}{n}\sum_{j=0}^{n-1}
\left(1-\frac{2C_1}{\widetilde{\rho}_{1,2}n}\right)^j
\lim_{n\to\infty}\frac{1}{n}\sum_{j=0}^{n-1}
\left(1-\frac{2C-2C_1}{\widetilde{\rho}_{1,2}n}\right)^j.
\end{aligned}
\end{equation}
Hence, the above monotonicity property follows from \eqref{7.41} and
\eqref{7.42}.
Notice, that the inequality in \eqref{7.42} is in fact the equality.

Similarly, on the base of relation \eqref{7.37} one can show the
monotonicity of the
other function $\eta(C)$.

All of these properties enables us to formulate the following general
result.

\begin{thm}
\label{thm31} The solution to the control problem is $\rho_1=1$ if and
only if the
minimum of the both functionals $J^{upper}$ and $J^{lower}$ is achieved
for $C=0$. In
this case, the minimum of objective function $J$ is given by
\eqref{7.14}. Otherwise,
there can be one of the following two cases for the solution to the
control problem.

(1) If the minimum of $J^{upper}$ is achieved for $C=0$, then the minimum
of $J^{lower}$
must be achieved for some positive value $C=\underline{C}$. Then the
solution to the
control problem is achieved for $\rho_1=1-\delta$, $\delta>0$, such that
$\delta n\to
\underline{C}$ as $n\to\infty$.

(2) If the minimum of $J^{lower}$ is achieved for $C=0$, then the minimum
of $J^{upper}$
must be achieved for some positive value $C=\overline{C}$. Then the
solution to the
control problem is achieved for $\rho_1=1+\delta$, $\delta>0$, such that
$\delta n\to
\overline{C}$ as $n\to\infty$.

The minimum of the functionals $J^{upper}$ and $J^{lower}$ can be found
by their
differentiating in $C$ and then equating these derivatives to zero.
\end{thm}

From Theorem \ref{thm31} we have the following property of the optimal
control.

\begin{cor}
\label{cor1} The solution to the control problem can be $\rho_1=1$ only
in the case
where
\begin{equation}
\label{7.43} j_1\leq j_2\frac{\rho_2}{1-\rho_2}.
\end{equation}
The equality in relation \eqref{7.43} is achieved only for $c_i\equiv c$,
$i=1,2,\ldots,n$, where $c$ is an any positive constant.
\end{cor}

Paper \cite{Abramov 2007 - water} discusses the case where the costs
$c_i$ have a linear
structure, and provides numerical analysis of this case. Specifically,
\cite{Abramov
2007 - water} finds numerically the relation between $j_1$ and $j_2$,
under which the
optimal solution to the control problem is $\rho_1=1$.

\section{Optimal policies of using water in large dams: Future research
problems}\label{Dam3}

In Sect. \ref{Dam1} and \ref{Dam2}, optimal policies of using
water in large dams are studied. Both these models are the same
queueing systems with specified dependence of a service on
queue-length, and in both of them the input stream is Poisson. The
only difference is in the criteria of minimization. The
minimization criteria in the simple model is based on penalties
for passage across the lower and higher level. The minimization
criteria of the extended model takes also into account the water
costs depending on the level of water in the dam.

However, the assumption that input stream is Poisson is
restrictive. So, the following extension of these models towards
their clearer application in practice lies in an assumption of a
wider class of input process than ordinary Poisson. In this
section we discuss a possible way of analysis in the case when the
input stream is a compound Poisson process.

Let $t_1, t_2,\ldots$ be the moment of arrivals of units at the
system, and let $\kappa_1, \kappa_2,\ldots$ be independent and
identically distributed integer random variables characterizing
the corresponding lengths of these units. As in Sect. \ref{Lost
messages} it is reasonable to assume that
$$\mathrm{P}\{\kappa^{lower}\leq\kappa_i\leq\kappa^{upper}\}=1$$
(see rel. \eqref{5.1}).

Let us recall the meaning of parameter $n$ in the dam models of
Sect. \ref{Dam1} and \ref{Dam2}. That parameter $n$ is defined in
the following formulation of queueing system. Consider a
single-server queueing system where the arrival flow of customers
is Poisson with rate $\lambda$ and the service time of a customer
depends upon queue-length as follows. If, at the moment the
customer's service beginning, the number of customers in the
system is not greater than $n$, then the service time of this
customer has the probability distribution $B_1(x)$. Otherwise (if
there are more than $n$ customers in the system at the moment of
the customer's service beginning), the probability distribution
function of the service time of this customer is $B_2(x)$.

In our case customers arrive by batches $\kappa_1, \kappa_2,
\ldots$, so as in Sect. \ref{Lost messages} it is reasonable to
define the random parameter $\zeta=\zeta(n)$:
$$
\zeta = \sup\left\{m: \sum_{i=1}^m\kappa_i\leq n\right\},
$$
There are two fixed values $\zeta^{lower}$ and $\zeta^{upper}$
depending on $n$ and
$$\mathrm{P}\{\zeta^{lower}\leq\kappa_i\leq\zeta^{upper}\}=1.$$

Thus, the following construction, similar to the construction that
used in Sect. \ref{Lost messages}, holds in this case as well.
Specifically, Definition \ref{defn5.1} and the following
application of the formulae for the total expectation for
characteristic of the system with random parameter $\zeta$ are the
same as above in Sect. \ref{Lost messages}. Then the definition of
random variables $T_\zeta^{(1)}$, $T_\zeta^{(2)}$,
$\nu_\zeta^{(1)}$, $\nu_\zeta^{(2)}$ with random parameter $\zeta$
is similar to the definition of the corresponding random variables
$T_n^{(1)}$, $T_n^{(2)}$, $\nu_n^{(1)}$, $\nu_n^{(2)}$ given in
Sect. \ref{Dam1} and \ref{Dam2}.

There are also equalities allying $\mathrm{E}\nu_{\zeta(n)}^{(1)}$
with $\mathrm{E}\nu_{\zeta(n)}^{(2)}$ and
$\mathrm{E}T_{\zeta(n)}^{(1)}$ with $\mathrm{E}T_{\zeta(n)}^{(2)}$
as
$$
\mathrm{E}\nu_{\zeta(n)}^{(2)}=a_{1,n}\mathrm{E}\nu_{\zeta(n)}^{(1)}+a_{2,n}
$$
and
$$
\mathrm{E}T_{\zeta(n)}^{(2)}=b_{1,n}\mathrm{E}T_{\zeta(n)}^{(1)}+b_{2,n}.
$$
The sequences $a_{1,n}$, $a_{2,n}$, $b_{1,n}$ and $b_{2,n}$
converge to the corresponding limits $a_1$, $a_2$, $b_1$ and
$b_2$. But the values of constants $a_1$, $a_2$, $b_1$ and $b_2$
are not the same as those given by \eqref{6.8} and \eqref{6.10}.
Special analysis based on renewal theory to find these constants
is required.

Another difference between the analysis of the problem described
in Sect. \ref{Dam2} and that analysis of the present problem is
connected with the structures of water costs. In the case of the
problem described in Sect. \ref{Dam2} the water costs are
non-increasing. In the present formulation, where the length of
arrival units are random, the water costs, being initially
non-increasing with respect to levels of water, after reduction to
the model with random parameter $\zeta$ become random variables.
However, the expected values of these random costs satisfy the
same property, they are non-increasing.

So, the statement on existence and uniqueness of an optimal
strategy of water consumption is anticipated to be the same as in
the earlier consideration in Sect. \ref{Dam2} where, however, the
optimal solution itself and its structure can be different.

\section{The buffer model with priorities}\label{The buffer model with
priorities}

In this section we give the application of Tak\'acs' theorem for a
specific buffer model
with priorities. This section contains one of the results of
\cite{Abramov 2006} related
to the effective bandwidth problem \cite{Abramov 2006}, and the model
studied here is a
particular case of more general models of \cite{Abramov 2006}. The
effective bandwidth
problem was a ``hot topic" of applications of queueing theory during the
decade
1990-2000. The detailed review of the existence literature up to the
publication time
can be found in the paper of Berger and Whitt \cite{Berger and Whitt
1998a}. For other
papers published later than aforementioned one see \cite{Berger and Whitt
1998b},
\cite{Bestimas et al 1998}, \cite{Courcoubetis Siris Stamoulis 1999},
\cite{Elwalid and
Mitra 1999}, \cite{Evans and Everitt 1999}, \cite{Kumaran et al 2000},
\cite{Lee et al
2005}, \cite{Wischik 1999}, \cite{Wischik 2001} and other papers.

Let $\mathcal{A}(t)$ be a renewal process of ordinary arrivals to
telecommunication
system having a large buffer of capacity $N$. Each of these arrivals
belongs to the
priority class $k$ with positive probability $p^{(k)}$, $\sum_{k=1}^\ell
p^{(k)}=1$,
where a smaller index corresponds to a higher priority. By thinning the
renewal process
$\mathcal{A}(t)$ we thus have $\ell$ independent renewal processes
$\mathcal{A}^{(1)}(t)$, $\mathcal{A}^{(2)}(t)$,\ldots,
$\mathcal{A}^{(\ell)}(t)$, where
the highest priority units are associated with the process
$\mathcal{A}^{(1)}(t)$, and
arrivals corresponding to $\mathcal{A}^{(i)}(t)$ have higher priority
than arrivals
corresponding to $\mathcal{A}^{(j)}(t)$ for $i<j$.

The departure process, $\mathcal{D}(t)$, is assumed to be a Poisson
process with
constant integer jumps $C\geq1$. The process $\mathcal{D}(t)$ is common
for all of units
independently of their priorities.

Let $\mathcal{Q}^{(k)}(t)$ denote the buffer content in time $t$ for
units of priority
$k$. All of the processes considered here are assumed to be right
continuous, having
left limits, and all of them are assumed to start at zero. Exceptions
from these rules
will be mentioned especially.

Let us now describe the priority rule. To simplify the explanation, let
us first assume
that the buffer is infinite. Then, for the highest priority units we
have:
\begin{equation}\label{8.1}
\mathcal{Q}^{(1)}(t)=\max\{0, \mathcal{Q}^{(1)}(t-)+\triangle
\mathcal{A}^{(1)}(t)-\triangle \mathcal{D}(t)\},
\end{equation}
where the triangle in \eqref{8.1} denotes the value of jump of the
process in point $t$,
i.e. $\triangle
\mathcal{A}^{(1)}(t)=\mathcal{A}^{(1)}(t)-\mathcal{A}^{(1)}(t-)$,
$\triangle \mathcal{D}(t)=\mathcal{D}(t)-\mathcal{D}(t-)$.

For the second priority units, we have:
\begin{equation}\label{8.2}
\mathcal{Q}^{(2)}(t)=\max\left\{0, \mathcal{Q}^{(2)}(t-)+\triangle
\mathcal{A}^{(2)}(t)-[\triangle
\mathcal{D}(t)-\mathcal{Q}^{(1)}(t-)]\mathrm{1}_{\{\mathcal{Q}^{(1)}(t)=0\}}\right\}.
\end{equation}
In general, for the $k+1$st priority units ($k=1,2,\ldots,\ell-1$) we
have
\begin{eqnarray}\label{8.3}
\mathcal{Q}^{(k+1)}(t)&=&\max\left\{0, \mathcal{Q}^{(k+1)}(t-)+\triangle
\mathcal{A}^{(k+1)}(t)\right.\\ &&-\left.\left[\triangle
\mathcal{D}(t)-\sum_{i=1}^k
\mathcal{Q}^{(i)}(t-)\right]\mathrm{1}_{\{\sum_{i=1}^k
\mathcal{Q}^{(i)}(t)=0\}}\right\}.\nonumber
\end{eqnarray}

Equations \eqref{8.2} and \eqref{8.3} implies that the priority rule is
the following.
The units are ordered and then leave the system according to their
priority as follows.
Let, for example, there be 6 units in total, two of them are of highest
priority and the
rest four are of second priority. Let $C=3$. Then, after the departure of
a group of
three units, there will remain only 3 units of the second priority. That
is all of the
units (i.e. two) of the highest priority and one unit of the second
priority leave
simultaneously.

For $k=1,2,\ldots,\ell$, using the notation

$$\mathcal{Q}_k(t)=\mathcal{Q}^{(1)}(t)+\mathcal{Q}^{(2)}+\ldots+\mathcal{Q}^{(k)},$$
and

$$\mathcal{A}_k(t)=\mathcal{A}^{(1)}(t)+\mathcal{A}^{(2)}+\ldots+\mathcal{A}^{(k)}$$
enables us to write the relations
\begin{equation}\label{8.4}
\mathcal{Q}_k(t)=\max\{0, \mathcal{Q}_k(t-)+\triangle
\mathcal{A}_k(t)-\triangle
\mathcal{D}(t)\},
\end{equation}
which are similar to \eqref{8.1}. The equivalence of \eqref{8.4} and
\eqref{8.1} -
\eqref{8.3} is proved by induction in \cite{Abramov 2006}.

According to \eqref{8.4}, the stability condition for this priority
queueing system with
infinite queues is $\rho_\ell<1$, where $\rho_\ell=\frac{\lambda}{\mu
C}$, $\lambda$ is
the expected number of arrivals per time unit, $\mu$ the parameter of
departure Poisson
process, and $C$ is the aforementioned constant jump of this Poisson
process. We will
assume that this condition holds for finite buffer models too (for finite
buffer models
the above stability condition is not required). So, if all the buffers
are large, then
the losses occur seldom.

In the case of finite buffer models, the representations are similar.
Denote the
capacity for the total number of units of priority $k$ (the number of
units of priority
$k$ that can be present simultaneously in the system) by $N^{(k)}$, and
$N^{(1)}+N^{(2)}+\ldots+N^{(\ell)}=N$. We will assume here that if upon
an arrival of a
batch, the buffer of the given class is overflowed, then the entire
arrival batch is
rejected and lost from the system.

Introduce new arrival processes $\overline{\mathcal{A}}^{(k)}(t)$,
$k=1,2,\ldots,\ell$,
corresponding to the $k$th priority as follows. Let the jump
$\triangle\overline{\mathcal{A}}^{(k)}(t)$ is defined as
\begin{equation}\label{8.5}
\triangle\overline{\mathcal{A}}^{(k)}(t)=\triangle{\mathcal{A}}^{(k)}(t)\mathrm{1}_{\{\mathcal{Q}^{(k)}(t)\leq
N^{(k)}\}}.
\end{equation}
The difference between the process $\mathcal{A}^{(k)}(t)$ and $\overline
{\mathcal{A}}^{(k)}(t)$ is only in jumps. In the case of the originally
defined process
$\mathcal{A}^{(k)}(t)$ the jumps are $\triangle \mathcal{A}^{(k)}(t)$,
while in the case
of the process $\overline{\mathcal{A}}^{(k)}(t)$ they are $\triangle
\overline{\mathcal{A}}^{(k)}(t)$. In addition, it turns out that the
process $\overline
{\mathcal{A}}^{(k)}(t)$ is not right continuous. There are isolated
points at moments of
overflowing the buffer of the given priority units.

Then the buffer content for units of the highest priority is defined by
equations:
\begin{eqnarray}
\mathcal{Q}^{(1)}(t)&=&\max\{\mathcal{Q}^{(1)}(t-)+\triangle
\mathcal{A}^{(1)}(t)-\triangle \mathcal{D}(t)\},\label{8.6}\\
\mathcal{Q}^{(1)}(t+)&=&\max\{\mathcal{Q}^{(1)}(t-)+\triangle
\overline{\mathcal{A}}^{(1)}(t)-\triangle \mathcal{D}(t)\}.\label{8.7}
\end{eqnarray}
In all continuity points of the processes $\mathcal{Q}^{(k)}(t)$,
$k=1,2,\ldots,\ell$,
one can obtain general representation similar to that of \eqref{8.4}.
Denoting

$$\overline{\mathcal{A}}_k(t)=\overline{\mathcal{A}}^{(1)}(t)+\overline{\mathcal{A}}^{(2)}(t)
+\ldots+\overline{\mathcal{A}}^{(k)}(t),$$
one can show the following representations:
\begin{eqnarray}
\mathcal{Q}_k(t)&=&\max\{\mathcal{Q}_k(t-)+\triangle
\mathcal{A}_k(t)-\triangle
\mathcal{D}(t)\},\label{8.8}\\
\mathcal{Q}_k(t+)&=&\max\{\mathcal{Q}_k(t-)+\triangle
\overline{\mathcal{A}}_k(t)-\triangle \mathcal{D}(t)\},\label{8.9}
\end{eqnarray}
which are supposed to be correct for continuity points of the processes
$\mathcal{Q}_k(t)$ only.

According to representations \eqref{8.8} and \eqref{8.9}, the cumulative
buffer contents
$\mathcal{Q}_k(t)$, $k=1,2,\ldots,\ell$, in continuity points behave as
usual
queue-length processes in $GI/M^C/1/N_k$ queueing systems, where
$N_k=N^{(1)}+N^{(2)}+\ldots+N^{(k)}$ ($N_\ell=N$). However, the behavior
of the number
of losses is essentially different. The losses in $GI/M^C/1/N_k$ queues
are not adequate
to the losses in the corresponding cumulative buffers $\mathcal{Q}_k(t)$.
Specifically,
the losses in $GI/M^C/1/N_k$ queues occur only in the case when the
buffer is overflowed
upon arrival of a customer meeting all of the waiting places busy. The
losses in the
cumulative buffers $\mathcal{Q}_k(t)$ can occur in many cases, when one
of the buffers
of priority units (say the $j$th buffer, $1 \leq j \leq k$) is
overflowed.

However, in some cases when the values $N_1 < N_2 < \ldots < N_\ell$ all
are large, the
correspondence between $GI/M^C/1/N_k$ queues and finite buffers models,
giving us useful
asymptotic result, is possible. Specifically, the loss probability of a
customer
arriving to one of the first $k$ buffers is not greater than $\pi_1 +
\pi_2 + \ldots +
\pi_k$, where $\pi_i$ denotes the loss probability in the corresponding
$GI/M^C/1/N_i$
queueing system.

Let $A_i(x)$ ($i=1,2,\ldots,\ell$) denote the probability distribution
function of
interarrival time between arrivals of one of the first $i$ priority
customers,
$\widehat{A}_i(s)$ is the Laplace-Stieltjes transform of $A_i(x)$. All of
the
probabilities $\pi_k$ are small and decrease geometrically fast (see
Theorem \ref{thm32}
below). Since the loads $\rho_1$, $\rho_2$,\ldots,$\rho_\ell$ of
corresponding
cumulative processes increase, i.e. $\rho_1<\rho_2<\ldots<\rho_\ell$,
then the roots
$\varphi_k$ of the corresponding functional equations appearing in the
aforementioned
theorem are ordered $\varphi_1<\varphi_2<\ldots<\varphi_k$. Such type of
dependence
between the roots $\varphi_k$ and the loads $\rho_k$, $k=1,2,\ldots,\ell$
is because of
the special construction of interarrival times. Then, under the
assumption that
$\varphi_j^{N_j}=o(\varphi_k^{N_k})$, $j<k$, (i.e. the losses of lower
priority
customers occur much often compared to those of higher priority) the loss
probability of
the cumulated buffer of the first $k$ priority customers can be
approximated by $\pi_k$.
In this case we have the following theorem.

\begin{thm}\label{thm32} For the loss probability $\pi_k$ of
cumulated buffer content we have the following estimation:
\begin{eqnarray}
\label{8.10}\\
\pi_k&=&\frac{(1-\rho_k)[1+C\mu\widehat{A}_k^\prime(\mu-\mu\varphi_k^C)]\varphi_k^{N_k}}
{(1-\rho_k)(1+\varphi_k+\varphi_2^2+\ldots+\varphi_k^{C-1})-\rho_k
[1+C\mu\widehat{A}_k^\prime(\mu-\mu\varphi_k^C)]\varphi_k^{N_k}}\nonumber\\
&&+o\left(\varphi_k^{2N_k}\right),\nonumber
\end{eqnarray}
where
\begin{eqnarray*}
\rho_k&=&\frac{\lambda_k}{\mu C},\\
\lambda_k&=&\lambda\sum_{i=1}^{k}p^{(i)},
\end{eqnarray*}
and $\varphi_k$ is the least positive root of the functional equation
$$
z=\widehat{A}_k(\mu-\mu z^C)
$$
in the interval (0,1).
\end{thm}

\begin{proof}
We consider the $GI/M^C/1/N_k$ queueing system and, similarly to the
proof of Sect.
\ref{GIM1n queue}, this proof is based on an application of Tak\'acs'
theorem. Following
Miyazawa \cite{Miyazawa 1990}, the loss probability for the
$GI/M^Y/1/N_k$ queueing
system (in this notation the batch size $Y$ is an integer random variable
rather than a
deterministic constant that in the cases of the notation $GI/M^C/1/N_k$
considered
before) is determined by the formula
\begin{equation}\label{8.11}
\pi_k=\frac{1}{\sum_{j=0}^{N_k}r_{k,j}},
\end{equation}
which is similar to relation \eqref{3.1}. The generating function of
$r_{k,j}$ is
\begin{equation}\label{8.12}
R_k(z)=\sum_{j=0}^\infty r_{k,j}z^j=\frac{(1-Y(z))\widehat{A}_k[\mu-\mu
Y(z)]}{\widehat{A}_k[\mu-\mu Y(z)]-z},
\end{equation}
where $Y(z)$ is the generating function of a complete batch size. In the
case of
$GI/M^C/1/N_k$ queueing system $Y(z)=z^C$, and \eqref{8.12} is rewritten
\begin{equation}\label{8.13}
R_k(z)=\frac{(1-z^C)\widehat{A}_k(\mu-\mu z^C)}{\widehat{A}_k(\mu-\mu
z^C)-z}.
\end{equation}
Expanding $(1-z^C)$ in the numerator of \eqref{8.13} as $1 - z^C = (1 -
z)(1 + z +
\ldots + z^{C-1})$, we have:
\begin{equation}\label{8.14}
R_k(z)=\frac{(1-z)(1+z+z^2+\ldots+z^{C-1})\widehat{A}_k(\mu-\mu
z^C)}{\widehat{A}_k(\mu-\mu z^C)-z}.
\end{equation}
Now, let us consider the other generating function given by
$\widetilde{R}_k(z)=\frac{1}{1-z}R(z)$. From \eqref{8.14} we obtain:
\begin{equation}\label{8.15}
\widetilde{R}_k(z)=\sum_{j=0}^\infty
\widetilde{r}_{k,j}z^j=\frac{(1+z+z^2+\ldots+z^{C-1})\widehat{A}_k(\mu-\mu
z^C)}{\widehat{A}_k(\mu-\mu z^C)-z},
\end{equation}
and the loss probability is
\begin{equation}\label{8.16}
\pi_k=\frac{1}{\widetilde{r}_{k,N_k}}.
\end{equation}
Now, application of Tak\'acs' theorem is straightforward, because the
term
$$
\frac{\widehat{A}_k(\mu-\mu z^C)}{\widehat{A}_k(\mu-\mu z^C)-z}
$$
of \eqref{8.15} has the representation similar to \eqref{1.2}, and
therefore the
corresponding coefficients of the generating function satisfy recurrence
relation
\eqref{1.1}.

For large $N_k$ we obtain:
\begin{eqnarray}\label{8.17}
\widetilde{r}_{k,N_k}&=&\frac{1+\varphi_k+\ldots+\varphi_k^{C-1}}{\varphi_k^{N_k}}\cdot\frac{1}{1+\mu
C\widehat{A}(\mu-\mu\varphi_k^C)}\\
&&+\frac{(1+\varphi_k+\ldots+\varphi_k^{C-1})\rho_k}{\rho_k-1}+o(1).\nonumber
\end{eqnarray}
Asymptotic relation \eqref{8.10} follows from \eqref{8.16} and
\eqref{8.17}.
\end{proof}

\section*{Acknowledgements}
The part of this paper has been written while the author was
visiting the Department of Industrial Engineering and Operations
Research in the University of California at Berkeley. Their
hospitality and good conditions for this work are appreciated. The
comments of the anonymous reviewer were helpful and stimulating as
well. The author also acknowledges with thanks the support of the
Australian Research Council, grant \#DP0771338.



\begin{thebibliography}{10}
\bibitem{Abdel-Hameed 2000}\textsc{Abdel-Hameed, M.} (2000). Optimal
control of a dam
using $P_{\lambda,\tau}^M$ policies and penalty cost when the input
process is a
compound Poisson process with positive drift. \emph{J. Appl. Probab.},
\textbf{37},
    408-416.

\bibitem{Abdel-Hameed and Nakhi 1990}\textsc{Abdel-Hameed, M. and Nakhi,
Y.} (1990).
Optimal control of a dam using $P_{\lambda,\tau}^M$ policies and penalty
cost: total
discounted and long-run average cases. \emph{J. Appl. Probab.},
\textbf{27},
    888-898.



\bibitem{Abramov 1981}\textsc{Abramov, V.M.} (1981). Some limit theorems
for a
single-channel system the intensity of which depends on queue-length.
\emph{Izv.
    Acad. Nauk SSSR. Techn. Kibernet.} (5), 53-57. (In Russian).

\bibitem{Abramov 1984}\textsc{Abramov, V.M.} (1984). Certain properties
of a stream of
lost calls. \emph{Engin. Cybernet.}, \textbf{22} (4), 135-137. Transl.
from:
\emph{Izv. Acad. Nauk SSSR. Techn. Kibernet.} (3), 148-150. (In Russian.)

\bibitem{Abramov 1991}\textsc{Abramov, V.M.} (1991). \emph{Investigation
of a Queueing
System with Service Depending on Queue-Length.} Donish, Dushanbe,
Tadzhikistan. (In
    Russian).

\bibitem{Abramov 1991-2}\textsc{Abramov, V.M.} (1991). On a property of
lost customers
for one queueing system with losses. \emph{Kibernetika} (Ukrainian
Academy of
    Sciences), (2), 123-124. (In Russian.)

\bibitem{Abramov 1997}\textsc{Abramov, V.M.} (1997). On a property of a
refusals stream.
    \emph{J. Appl. Probab.}, \textbf{34}, 800-805.

\bibitem{Abramov 2001}\textsc{Abramov, V.M.} (2001). On losses in
$M^X/GI/1/n$ queues.
    \emph{J. Appl. Probab.}, \textbf{38}, 1079-1080.

\bibitem{Abramov 2002}\textsc{Abramov, V.M.} (2002). Asymptotic analysis
of the
$GI/M/1/n$ loss system as $n$ increases to infinity. \emph{Ann. Operat.
Res.},
    \textbf{112}, 35-41.

\bibitem{Abramov 2004}\textsc{Abramov, V.M.} (2004). Asymptotic behaviour
of the number
    of lost messages. \emph{SIAM J. Appl. Math.}, \textbf{64}, 746-761.


\bibitem{Abramov 2007-1}\textsc{Abramov, V.M.} (2007). Asymptotic
analysis of loss
probabilities in $GI/M/m/n$ queueing systems as $n$ increases to
infinity.
    \emph{Qual. Technol. Quantitat. Manag.}, \textbf{4}, 379-393.

\bibitem{Abramov 2007-2}\textsc{Abramov, V.M.} (2007). Optimal control of
a large dam.
    \emph{J. Appl. Probab.}, \textbf{44}, 249-258.

\bibitem{Abramov 2007 - water}\textsc{Abramov, V.M.} (2007). Optimal
control of a large
dam, taking into account the water costs. \emph{arXiv:math 0701458v3}.

\bibitem{Abramov 2007-3}\textsc{Abramov, V.M.} (2007). Losses in
$M/GI/m/n$ queues.
    \emph{arXiv:math 0506033}.

\bibitem{Abramov 2006}\textsc{Abramov, V.M.} (2008). The effective
bandwidth problem
    revisited. \emph{Stoch. Models}, \textbf{24}, to appear.

\bibitem{Ait-Hellal et al 1999}\textsc{Ait-Hellal, O., Altman, E.,
Jean-Marie, A. and
Kurkova, I.A.} (1999). On loss probabilities in presence of redundant
packets and
several traffic sources. \emph{Perform. Eval.}, \textbf{36-37}, 485-518.

\bibitem{Altman and Jean-Marie 1998}\textsc{Altman, E. and Jean-Marie,
A.} (1998). Loss
probabilities for messages with redundant packets feeding a finite
buffer.
    \emph{IEEE J. Select. Areas Communicat.}, \textbf{16}, 779-787.



\bibitem{Azlarov and Tahirov 1974}\textsc{Azlarov, T.A. and Takhirov, A.}
(1974). The
limiting distributions for a single-channel system with limited number of
waiting
places. \emph{Engin. Cybernet.}, \textbf{12} (5), 47-51. Transl. from:
\emph{Izv.
    Acad. Nauk SSSR. Techn. Kibernet.} (5), 53-57. (Russian.)

\bibitem{Azlarov and Tashmanov 1974}\textsc{Azlarov, T.A. and Tashmanov,
H.T.} (1974).
Two-dimensional distributions in the queueing system $M/G/1/N$. In
\emph{Random
Processes and Statistical Inference}, (4), 14-25. Fan, Tashkent, Uzbek.
SSR.

\bibitem{Bae Kim and Lee 2003}\textsc{Bae, J., Kim, S. and Lee, E.Y.}
(2003). Average
cost under the $P_{\lambda,\tau}^M$ policy in a finite dam with compound
Poisson
    inputs. \emph{J. Appl. Probab.}, \textbf{40}, 519-526.

\bibitem{Berger and Whitt 1998a}\textsc{Berger, A.W. and Whitt, W.}
(1998). Effective
bandwidths with priorities. \emph{IEEE/ACM Trans. Network.}, \textbf{6}
(4),
    447-460.

\bibitem{Berger and Whitt 1998b}\textsc{Berger, A.W. and Whitt, W.}
(1998). Extending
the effective bandwidth concept to networks with priority classes.
\emph{IEEE
    Communicat. Magaz.}, \textbf{36}, (8) 78-83.

\bibitem{Bertrand 1887}\textsc{Bertrand, J.} (1887). Solution d'un
probl\'eme.
\emph{Comptes Rendus de l'Acad\'emie des Sciences, Paris} \textbf{105},
407.

\bibitem{Bestimas et al 1998}\textsc{Bertsimas, D., Paschalidis, I.C. and
Tsitsiklis,
J.N.} (1998). Asymptotic buffer overflow probabilities in multiclass
multiplexers:
An optimal control approach. \emph{IEEE Trans. on Automat. Contr.}
\textbf{43},
    315-335.


\bibitem{Choi and Kim 2000}\textsc{Choi, B.D. and Kim, B.} (2000). Sharp
results on
convergence rates for the distribution of $GI/M/1/K$ queues as $K$ tends
to
    infinity. \emph{J. Appl. Probab.}, \textbf{37}, 1010-1019.

\bibitem{Choi Kim and Wee 2000}\textsc{Choi, B.D., Kim, B. and Wee,
I.-S.} (2000).
Asymptotic behavior of loss probability in $GI/M/1/K$ queue as $K$ tends
to
    infinity. \emph{Queueing Syst.}, \textbf{36}, 437-442.

\bibitem{Choi et al 2003}\textsc{Choi, B.D., Kim, B., Kim, J. and Wee,
I.-S.} (2003).
Exact convergence rate for the distributions of $GI/M/c/K$ queue as $K$
tends to
    infinity. \emph{Queueing Syst.}, \textbf{44}, 125-136.

\bibitem{Chydzinski 2004}\textsc{Chydzinski, A.} (2004). On the remaining
service time
upon reaching a given level in $M/G/1$ queues. \emph{Queueing Syst.},
\textbf{47},
    71-80.

\bibitem{Chydzinski 2005}\textsc{Chydzinski, A.} (2005). On the
distribution of
consecutive losses in a finite capacity queue. \emph{WSEAS Trans.
Curcuits. Syst.},
    \textbf{4}, 117-124.

\bibitem{Chydzinski 2007}\textsc{Chydzinski, A.} (2007). Buffer overflow
period in a MAP
queue. \emph{Math. Probl. Engineer.}, \textbf{2007}, Article ID: 34631,
18p.

\bibitem{Cidon Khamisy and Sidi 1993}\textsc{Cidon, I., Khamisy, A. and
Sidi, M.}
(1993). Analysis of packet loss process in high speed networks.
\emph{IEEE Trans.
    Inform. Theor.}, \textbf{39}, 98-108.

\bibitem{Cohen 1971}\textsc{Cohen, J.W.} (1971). On the busy periods for
the $M/G/1$
queue with finite and with infinite waiting room. \emph{J. Appl.
Probab.},
    \textbf{8}, 821-827.

\bibitem{Cooper and Tilt 1976}\textsc{Cooper, R.B. and Tilt, B.} (1976).
On the
relationship between the distribution of maximal queue-length in the
$M/G/1$ queue
and the mean busy period in the $M/G/1/n$ queue. \emph{J. Appl. Probab.},
    \textbf{13}, 195-199.

\bibitem{Courcoubetis Siris Stamoulis 1999}\textsc{Courcoubetis, C.,
Siris, V.A. and
Stamoulis, G.} (1999). Application of many sources asymptotic and
effective
bandwidth for traffic engineering. \emph{Telecommun. Syst.}, \textbf{12},
167-191.

\bibitem{DeBoer et al 2001}\textsc{De Boer, P.-T., Nicola, F.P. and Van
Ommeren, J.C.W.}
(2001). The remaining service time upon reaching a high level in $M/G/1$
queues.
    \emph{Queueing Systems}, \textbf{39}, 55-78.

\bibitem{Dube Ait-Hellal and Altman 2003}\textsc{Dube, P., Ait-Hellal, O.
and Altman,
E.} (2003). On loss probabilities in presence of redundant packets with
random drop.
    \emph{Perform. Eval.}, \textbf{53}, 147-167.


\bibitem{Dube and Altman 2003}\textsc{Dube, P. and Altman, E.} (2003).
Queueing and
fluid analysis of partial message discarding policy. \emph{Queueing
Syst.},
    \textbf{44}, 253-280.

\bibitem{Elwalid and Mitra 1999}\textsc{Elwalid, A.I. and Mitra, D.}
(1999). Design of
generalized processor sharing schedulers with statistically multiplex
heterogeneous
    QoS classes. \emph{Proc. IEEE INFOCOM'99}, 1220-1230.


\bibitem{Evans and Everitt 1999}\textsc{Evans, J.S. and Everitt, D.}
(1999). Effective
bandwidth-based admission control for multiservice CDMA cellular
networks.
    \emph{IEEE Trans. Vehicul. Technol.}, \textbf{48}, 36-46.

\bibitem{Faddy 1974}\textsc{Faddy, M.J.} (1974). Optimal control of
finite dams:
discrete (2-stage) output procedure. \emph{J. Appl. Probab.},
\textbf{11}, 111-121.


\bibitem{Feller 1966}\textsc{Feller, W.} (1966). \emph{An Introduction to
Probability
    Theory and Its Applications.} Vol. 2, John Wiley, New York.

\bibitem{Gurewitz Sidi and Cidon 2000}\textsc{Gurewitz, O., Sidi, M. and
Cidon, I.}
(2000). The ballot theorem strikes again: Packet loss process
distribution.
    \emph{IEEE Trans. Inform. Theor.}, \textbf{46}, 2588-2595.

\bibitem{Hardy 1949}\textsc{Hardy, G.H.} (1949). \emph{Divergent Series},
Oxford
    University Press, Oxford.


\bibitem{Hardy and Littlewood 1929}\textsc{Hardy, G.H. and Littlewood,
J.E.} (1929).
Notes on the theory of series (XI): On Tauberian theorems. \emph{Proc.
London Math.
    Soc.}, Ser. 2, \textbf{30}, 23-37.

\bibitem{Hardy Littlewood and Polya 1952}\textsc{Hardy, G.H., Littlewood,
J.E. and
Polya, G.} (1952). \emph{Inequalities}, Cambridge University Press,
London.

\bibitem{Heyde 1969} \textsc{Heyde, C.C.} (1969). A derivation of the
ballot theorem
from the Spitzer-Pollaczek identity. \emph{Proc. Cambridge Phil. Soc.},
\textbf{65},
    755-757.

\bibitem{Kemp and Kemp 1968}\textsc{Kemp, A.W. and Kemp, C.D.} (1968). On
a distribution
associated with certain stochastic processes. \emph{J. Royal Stat. Soc.},
Ser. B.
    \textbf{30}, 160-163.


\bibitem{Kim and Choi 2003}\textsc{Kim, B. and Choi, B.D.} (2003).
Asymptotic analysis
and simple approximation of the loss probability in the $GI^X/M/c/K$
queue.
    \emph{Perform. Eval.}, \textbf{54}, 331-356.

\bibitem{Krattenthaler and Mohanty 1994}\textsc{Krattenthaler, C. and
Mohanty, S.G.}
(1994). $q$-Generalization of a ballot theorem. \emph{Discrete Math.},
\textsc{126},
    195-208.

\bibitem{Kumaran et al 2000}\textsc{Kumaran, K., Margrave, G.E., Mitra,
D. and Stanley,
K.R.} (2000). Novel techniques for the design and control of generalized
processor-sharing schedulers for multiple QoS classes. \emph{Proc.
INFOCOM'00},
    \textbf{2}, 932-941.


\bibitem{Lam and Lou 1987}\textsc{Lam, Y. and Lou, J.H.} (1987). Optimal
control for a
    finite dam. \emph{J. Appl. Probab.}, \textbf{24}, 196-199.

\bibitem{Lee and Ahn 1998}\textsc{Lee, E.Y. and Ahn, S.K.} (1998).
$P_\tau^M$ policy for
a dam with input formed by compouned Poisson process. \emph{J. Appl.
Probab.},
    \textbf{35}, 482-488.

\bibitem{Lee et al 2005}\textsc{Lee, J.Y., Kim, S., Kim, D. and Sung,
D.K.} (2005).
Bandwidth optimization for internet traffic in generalized
processor-sharing
servers. \emph{IEEE Trans. Paral. Distrib. Syst.}, \textbf{16}, 324-334.

\bibitem{Lefevre 2007}\textsc{Lef\'evre, C.} (2007). First-crossing and
ballot type
results for some non-stationary sequences. \emph{Adv. Appl. Probab.},
\textbf{39},
    492-509.

\bibitem{Lefevre and Loisel 2008}\textsc{Lef\'evre, C. and Loisel, S.}
(2008). On finite
time ruin probabilities for classical risk models. \emph{Scand. Actuar.
J.},
    \textbf{2008}, 41-60.

\bibitem{Marschall and Olkin 1979}\textsc{Marschall, A.W. and Olkin, I.}
(1979).
\emph{Inequalities: Theory of Majorizations and Its Applications}.
Academic Press,
    London.

\bibitem{Mazza and Rulliere 2004}\textsc{Mazza, C. and Rulli\'ere, D.}
(2004). A link
between wave governed random motions and ruin processes. \emph{Insurance:
Math. and
    Econom.}, \textbf{35}, 205-222.

\bibitem{Mendelson 1982}\textsc{Mendelson, H.} (1982). A batch-ballot
problem and
    applications. \emph{J. Appl. Probab.}, \textsc{19}, 144-157.


\bibitem{Miyazawa 1990}\textsc{Miyazawa, M.} (1990). Complementary
generating function
for the $M^X/GI/1/k$ and $GI/M^Y/1/k$ queues and their application to the
comparison
for loss probabilities. \emph{J. Appl. Probab.}, \textbf{27}, 684-692.

\bibitem{Pacheco and Ribeiro 2006}\textsc{Pacheco, A. and Ribeiro, H.}
(2006).
Consecutive customers loss probabilities in $M/G/1/n$ and $GI/M(m)//n$
systems. In:
\emph{Proc. Workshop on Tools for Solving Structured Markov Chains},
Pisa, Italy,
    October, 2006.

\bibitem{Pacheco and Ribeiro 2008}\textsc{Pacheco, A. and Ribeiro, H.}
(2008).
Consecutive customers losses in regular and oscilating $M^X/GI/1/n$
systems.
    \emph{Queueing Syst.}, \textbf{58}, 121-136.

\bibitem{Pacheco and Ribeiro 2008b}\textsc{Pacheco, A. and Ribeiro, H.}
(2008).
Consecutive customers losses in oscilating $GI^X/M//n$ systems with state
dependent
    service rates. \emph{Ann. Operat. Res.}, \textbf{162}, 143-158.


\bibitem{Pekoz 1999}\textsc{Pek\"oz, E.} (1999). On the number of
refusals in a busy
    period. \emph{Probab. Engin. Inform. Sci.}, \textbf{13}, 71-74.

\bibitem{Pekoz Righter and Xia 2003}\textsc{Pek\"oz, E., Righter, R. and
Xia, C.H.}
(2003). Characterizing losses in finite buffer systems. \emph{J. Appl.
Probab.},
    \textbf{40}, 242-249.

\bibitem{Postnikov 1980}\textsc{Postnikov, A.G.} (1980). Tauberian theory
and its
application. \emph{Proc. Steklov Math. Inst.}, \textbf{144}(2), 1-138.

\bibitem{Righter 1999}\textsc{Righter, R.} (1999). A note on losses in
$M/GI/1/n$
    queues. \emph{J. Appl. Probab.}, \textbf{36}, 1240-1243.

\bibitem{Rosenlund 1973}\textsc{Rosenlund, S.I.} (1973). On the length
and number of
served customers of the busy period of the generalized $M/G/1$ queue with
finite
    waiting room. \emph{Adv. Appl. Probab.}, \textbf{5}, 379-389.

\bibitem{Ross 2000}\textsc{Ross, S.M.} (2000). \emph{Introduction to
Probability
    Models}, 7th edn. Academic Press, Burlington, MA.

\bibitem{Sheinin 1994}\textsc{Sheinin, O.} (1994). Bertrand's work on
probability.
    \emph{Archiv for History of Exact Sciences}, \textbf{48}, 155-199.


\bibitem{Simonot 1998}\textsc{Simonot, F.} (1998). A comparison of the
stationary
distributions of $GI/M/c/n$ and $GI/M/c$. \emph{J. Appl. Probab.},
\textbf{35},
    510-515.

\bibitem{Subhankulov 1976}\textsc{Subhankulov, M.A.} (1976).
\emph{Tauberian Theorems
    with Remainder}, Nauka, Moscow. (In Russian.)

\bibitem{Szekely 1986}\textsc{Szekely, G.J.} (1986). \emph{Paradoxes in
Probability
    Theory and Mathematical Statistics.} Reidel.

\bibitem{Sznajder and Filar 1992}\textsc{Sznajder, R. and Filar, J.A.}
(1992). Some
comments of a theorem of Hardy and Littlewood. \emph{J. Optimizat. Theor.
Appl.},
    \textbf{75}, 71-78.

\bibitem{Takacs 1962}\textsc{Tak\'acs, L.} (1962). \emph{Introduction to
the Theory of
    Queues}. Oxford University Press, New York/London.

\bibitem{Takacs 1967}\textsc{Tak\'acs, L.} (1967). \emph{Combinatorial
Methods in the
    Theory of Stochastic Processes}, John Wiley, New York.

\bibitem{Takacs 1968}\textsc{Tak\'acs, L.} (1968). On dams of finite
capacity. \emph{J.
    Austral. Math. Soc.}, \textbf{8}, 161-170.

\bibitem{Takacs 1969}\textsc{Tak\'acs, L.} (1969). On the classic ruin
problem. \emph{J.
    Amer. Statist. Assoc.}, \textbf{64}, 889-906.

\bibitem{Takacs 1969-2}\textsc{Tak\'acs, L.} (1969). On inverse queueing
processes.
    \emph{Zastos. Mat.}, \textbf{10}, 213-224.

\bibitem{Takacs 1970}\textsc{Tak\'acs, L.} (1970). On the distribution of
the maximum of
sum of mutually independent and identically distributed random variables.
\emph{Adv.
    Appl. Probab.}, \textbf{2}, 344-354.

\bibitem{Takacs 1970-2}\textsc{Tak\'acs, L.} (1970). On the distribution
of supremum of
stochastic processes. \emph{Ann. Inst. H. Poincar\'e}, Sect. B. (N.S.)
\textbf{6},
    237-247.

\bibitem{Takacs 1974}\textsc{Tak\'acs, L.} (1974). A single-server queue
with limited
    virtual waiting time. \emph{J. Appl. Probab.}, \textbf{11}, 612-617.

\bibitem{Takacs 1975}\textsc{Tak\'acs, L.} (1975). Combinatorial and
analytic methods in
the theory of queues. \emph{Adv. Appl. Probab.}, \textbf{7}, 607-635.


\bibitem{Takacs 1976}\textsc{Tak\'acs, L.} (1976). On the busy periods of
single-server
queues with Poisson input and generally distributed service times.
\emph{Operat.
    Res.}, \textbf{24}, 564-571.

\bibitem{Takacs 1989}\textsc{Tak\'acs, L.} (1989). Ballots, queues and
random graphs.
    \emph{J. Appl. Probab.}, \textbf{26}, 103-112.

\bibitem{Takacs 1997}\textsc{Tak\'acs, L.} (1997). On the ballot
theorems. In:
\emph{Advances in Combinatorical Methods and Application in Probability
and
    Statistics}, N.Balakrishnan ed., Birkh\"auser, 1997, pp. 97-114.

\bibitem{Tahirov 1980}\textsc{Takhirov, A.} (1980). Joint distributions
of certain
characteristics of a single-line queueing system with bounded queue.
\emph{Izv. AN
    SSSR. Techn. Kibernet.} (3), 84-91. (In Russian).

\bibitem{Tamaki 2001}\textsc{Tamaki, M.} (2001). Optimal stopping on
trajectories and
    the ballot problem. \emph{J. Appl. Probab.}, \textbf{38}, 946-959.

\bibitem{Tomko 1967}\textsc{Tomk\'o, J.} (1967). One limit theorem in
queueing problem
as input rate increases indefinitely. \emph{Studia Sci. Math. Hungar.}
\textbf{2},
    447-454. (In Russian.)

\bibitem{Widder 1941}\textsc{Widder, D.V.} (1941). \emph{The Laplace
Transform}.
    Princeton University Press, Princeton, NJ.

\bibitem{Whitt 2004}\textsc{Whitt, W.} (2004). Heavy-traffic limits for
loss proportions
in single-server queues. \emph{Queueing Syst.}, \textbf{46}, 507-536.

\bibitem{Willmot 1988}\textsc{Willmot, G.E.} (1988). A note on the
equilibrium $M/G/1$
    queue-length. \emph{J. Appl. Probab.}, \textbf{25}, 228-231.

\bibitem{Wischik 1999}\textsc{Wischik, D.} (1999). The output of switch,
or, effective
    bandwidth for networks. \emph{Queueing Syst.}, \textbf{32}, 383-396.

\bibitem{Wischik 2001}\textsc{Wischik, D.} (2001). Sample path large
deviations for
queues with many outputs. \emph{Ann. Appl. Probab.}, \textbf{11},
379-404.

\bibitem{Wolff 1989}\textsc{Wolff, R.} (1989). \emph{Stochastic Modelling and the Theory of Queues}. Prentice Hall, Englewood Cliffs, NJ.


\bibitem{Wolff 2002}\textsc{Wolff, R.} (2002). Losses per cycle in a
single-server
    queue. \emph{J. Appl. Probab.}, \textbf{39}, 905-909.

\bibitem{Zukerman 1977}\textsc{Zukerman, D.} (1977). Two-stage output
procedure of a
    finite dam. \emph{J. Appl. Probab.}, \textbf{14}, 421-425.


\end{thebibliography}
\end{document}